\documentclass[11pt,a4paper]{article}
\usepackage[utf8]{inputenc}
\usepackage[T1]{fontenc}

\usepackage{vmargin}
\usepackage{amsmath}
\usepackage{amsthm}
\usepackage{amssymb}
\usepackage{amsopn}
\usepackage{enumitem}
\usepackage{color}
\usepackage{esint}
\usepackage{mathtools}
\usepackage{hyperref}
\usepackage{doi}

\DeclareMathOperator{\divo}{div}
\DeclareMathOperator{\Id}{Id}

\DeclareMathOperator{\dist}{dist}

\DeclareMathOperator{\co}{co}
\DeclareMathOperator{\spano}{span}
\DeclareMathOperator{\supp}{supp}

\DeclareMathOperator{\sinto}{sint}
\DeclareMathOperator{\into}{int}

\theoremstyle{plain}
\newtheorem{thm}{Theorem}[section]
\newtheorem{prop}[thm]{Proposition}
\newtheorem{cor}[thm]{Corollary}
\newtheorem{lem}[thm]{Lemma} 

\theoremstyle{definition}

\theoremstyle{remark}
\newtheorem{rem}[thm]{Remark}

\newcommand{\R}{\mathbb{R}}

\newcommand{\Z}{\mathbb{Z}}

\newcommand{\N}{\mathbb{N}}
\newcommand{\wto}{\rightharpoonup}

\newcommand{\css}{\subset\subset}

\newcommand{\eps}{\varepsilon}

\begin{document}
\begin{center}
\begin{Large}
Connecting Atomistic and Continuous Models of Elastodynamics
\end{Large}
\\[0.5cm]
\begin{large}
Julian Braun\footnote{Universität Augsburg, Germany, {\tt julian.braun@math.uni-augsburg.de}}
\end{large}
\\[0.5cm]
\today
\\[1cm]
\end{center}
\begin{abstract}
We prove long-time existence of solutions for the equations of atomistic elastodynamics on a bounded domain with time-dependent boundary values as well as their convergence to a solution of continuum nonlinear elastodynamics as the interatomic distances tend to zero. Here, the continuum energy density is given by the Cauchy-Born rule. The models considered allow for general finite range interactions. To control the stability of large deformations we also prove a new atomistic Gårding inequality.
\end{abstract}

\section{Introduction}

The dynamic behavior of an elastic material is classically described by continuum mechanics in terms of a deformation mapping that satisfies the second-order, nonlinear, hyperbolic partial differential equations of elastodynamics subject to given initial and boundary conditions. The precise equations are given by Newton's second law of motion. For hyperelastic materials, the internal forces are obtainable as the first variation of an elastic energy that depends in a local but nonlinear way on the deformation gradient.

At the same time, on a microscopic level, crystalline solids consist of many atoms, e.g. on a part of a Bravais lattice, and can be described directly by their interaction. The interatomic forces can effectively be modeled in terms of classical interaction potentials. Using again Newton's second law of motion, we arrive at a very high dimensional system of ordinary differential equations.

The classical connection between atomistic and continuum models of nonlinear elasticity is provided by the Cauchy-Born rule: The continuum stored energy function associated to a macroscopic affine map is given by the energy per unit volume of a crystal which is homogeneously deformed with the same affine mapping. In particular, this entails the assumption that there are no fine scale oscillations on the atomistic scale. We will call this function the Cauchy-Born energy density in the following. Note though, that it is not clear a priori whether the Cauchy-Born hypothesis is true or not.

In the previous work \cite{braun16static}, Schmidt and the author rigorously discuss existence and convergence of solutions as well as the Cauchy-Born rule in the case of elastostatics. We also refer to the introduction of \cite{braun16static} for a more exhaustive account of recent mathematical progress in this field with an emphasis on static equilibrium problems.

Our aim in this work is to establish a rigorous link between atomistic models and the corresponding Cauchy-Born continuum models for the elastodynamic behavior of crystalline solids accounting for body forces, boundary values, and initial conditions. We will prove such a connection in the asymptotic regime where the interatomic distance $\eps$ goes to $0$ and will even consider long times and large deformations.

In more detail, we will show that as long as the continuum solution exists and satisfies certain stability conditions, there are solutions of the corresponding atomistic initial-boundary value problems with lattice spacing $\eps$ that converge to the continuum solution uniform in time as $\eps \to 0$. Or to look at it from the other direction: We will give sufficient conditions on the body forces, initial conditions, and boundary conditions in the atomistic model such that there are solutions that follow a continuum solution as $\eps \to 0$ and thus obey the Cauchy-Born rule.

Recent publications (\cite{emingstatic}, \cite{emingdynamic}, \cite{ortnertheil13}) already provide results of this type for small displacements on a flat torus and for the full space problem with a far-field condition, respectively. But naturally the question arises if such an analysis is possible for a material occupying a general finite domain, on the boundary of which there might also be prescribed time-dependent boundary values. To cite Ericksen \cite[p.~207]{Ericksen08} ``Cannot someone do something like this for a more realistic case, say zero surface tractions on part of the boundary and given displacements on the remainder?'' In \cite{braun16static}, Schmidt and the author have given a positive answer to this question for the static problem with given displacements on the full boundary. Implicitly, results under mixed or pure traction boundary conditions are largely included. Here, we want to extend these results to the dynamic case. For this we will make use of some of the estimates proven in \cite{braun16static}.

Such a treatment of arbitrary domains and general displacement boundary conditions is of interest not only from a theoretical perspective but also with a view to specific situations that are of interest in applications, that can use our results as a starting point. Besides discussing elastic behavior, our results can also be used to discuss questions of stability and even the onset of instabilities, since the stability assumptions we make are designed to be quite sharp.

While equilibrium situations play an important role in mechanics, in many situations the material behavior decisively depends on inertial effects and, as a consequence, static or quasistatic descriptions are insufficient. Mathematically, the dynamic equations are considerably more challenging. Already in the continuum description, we have to discuss a quasi-linear, second-order, hyperbolic system under time-dependent boundary conditions. While the problems have natural energies that are conserved, they are of limited use in the analysis since level sets are typically unbounded in relevant norms and can contain regions where there is a loss of stability. 

In view of the recent results by other authors mentioned above, it should be pointed out that the treatment of the boundary value case is not a straightforward extension of the previous results. Let us just mention some of the difficulties. An important but subtle point in our main theorem is the condition posed on the atomistic boundary data. The atomistic boundary values can not be arbitrary but have to be chosen in a precise range to reflect the continuum boundary values while at the same time ensuring that there are no surface effects, so that the Cauchy-Born rules holds up to the boundary. Indeed, for arbitrary atomistic boundary values, one expects surface effects and a failure of the Cauchy-Born rule. A precise and rigorous mathematical treatment of such surface relaxation effects is currently still out of reach (but cf.\ \cite{theilsurface11}). In order to allow for as many atomistic boundary conditions (and body forces) as possible, we consider general convergence rates $\eps^\gamma$ in our main theorem, Theorem \ref{thm:atomisticwavethm}, and only restrict $\gamma$ as much as necessary. While smaller $\gamma$ will lead to a larger variety of atomistic boundary values, the maximal $\gamma = 2$ gives optimal convergence rates. Furthermore, there are several more technical problems. Most importantly, certain methods which are available on the flat torus or on the whole space do not translate to our setting. E.g., quasi-interpolations do not preserve boundary conditions. Instead, we use a different and more robust approach to the residual estimates that works in all cases and does not require any more regularity than the dynamical result in \cite{ortnertheil13}.

For the atomistic equations of elastodynamics, besides considering boundary conditions, there is an additional open question of considerable importance. Is it possible to establish the existence of atomistic solutions that satisfy the Cauchy-Born rule and their link to the continuum solutions not just for small times and deformations close to a stable affine configuration but for long times and large deformations?

In \cite{emingdynamic}, E and Ming only prove a short time result. An extension is not obvious, since their methods were restricted to small displacements. In \cite{ortnertheil13}, Ortner and Theil are indeed aware of this restriction that also applies to their results. They proposed that one could indeed extend the results to long times if one were to establish an atomistic version of the Gårding inequality.

The need for such an inequality can be understood as follows. For small displacements it is sufficient to stay close enough to a given stable affine deformation to ensure that the second variation of the potential energy is positive uniformly in $H^1_0$. For large deformations this is, in general, false. Even if at each point the gradient corresponds locally to a stable affine deformation, the second variation can still be negative globally. A Gårding inequality helps to work around this situation. It states that one can still get uniform positivity in $H^1_0$ from the local stability if one is willing to add a large constant times a lower order term (more precisely, the square of the $L^2$-norm). While a Gårding inequality alone is not sufficient to treat large deformations for the static equations, this inequality is key to the dynamic equations. In the continuum case the Gårding inequality is indeed part of a well-established theory. In the discrete case such a Gårding inequality is more subtle. Already the question of what constitutes a `locally stable deformation' requires a deeper analysis in the atomistic case, as we will discuss in Section \ref{sec:preparations}. In particular, the local stability assumption from the continuum case turns out to be insufficient. Additionally, in the continuous case the continuity of the coefficients and the deformation gradient is a crucial assumption for the Gårding inequality. This assumption has to be replaced by a more quantified version that is adapted to the discrete nature of the atomistic problem. In this spirit we will indeed establish an atomistic Gårding inequality, Theorem \ref{thm:discreteGårding}.

To give a short overview, we will start by giving a precise description of our models in Section \ref{sec:models}. Next we will shortly discuss in Section \ref{sec:preparations} the different concepts of stability and cite some important results about the stability constants. A crucial ingredient in the proof of our main theorem will be the observation that a smooth solution of the Cauchy-Born continuum equations solves the atomistic equations up to a small residuum. Still in Section \ref{sec:preparations}, this is made rigorous with the residual estimates that we can cite from \cite{braun16static}. In Section \ref{sec:cont}, we will use existing short time results, as well as different ideas about hyperbolic regularity, optimal elliptic regularity under weak assumptions, and a fine discussion of compositions and products in Sobolev spaces, to prove a result on the maximal existence time of solutions to the continuum equations of elastodynamics.

But the main results in this work can be found in Section \ref{sec:atom} where we will state and prove the atomistic Gårding inequality, Theorem \ref{thm:discreteGårding}, as well as our main theorem, Theorem \ref{thm:atomisticwavethm}. It states that as long as the continuum solutions exists and is atomistically stable, there exists a solution to the atomistic equations close to it. We also quantify the required conditions on the atomistic body forces, boundary values, and initial values in relation to their continuum counterparts.

Lastly, in the appendices we collect and prove a few technical lemmata concerning elliptic and hyperbolic regularity as well as the multiplication of many Sobolev functions. Some of these results might already be (implicitly) known to experts in the field but they do not seem to be available in the literature.

\section{The Models} \label{sec:models}
\subsection{The Continuum Model}
We consider a bounded, open set $\Omega \subset \R^d$, a time interval $[0,T)$, deformations $y \colon \Omega \times [0,T) \to \R^d$, an energy density $W_{\rm cont} \colon \R^{d \times d} \to (-\infty,\infty]$, initial positions $h_0 \in H^1(\Omega; \R^d)$, initial velocities $h_1 \in L^2(\Omega;\R^d)$, a body force $f \in L^2(\Omega \times [0,T); \R^d)$, and Dirichlet boundary data $g \in L^2([0,T); H^1(\Omega ; \R^d))$. We then consider the potential energy
\[
E(y;f)(t) = \int\limits_\Omega W_{\rm cont}(\nabla y (x,t)) - y(x,t)f(x,t) \,dx,
\]
whenever it is well-defined.

In the static case the relevant deformations are the local minimizers of the potential energy. In the dynamic case, we use Newton's second law of motion where the forces are given by the first variation of the potential energy. The reference body is assumed to have constant density $\rho$. By choice of units we can just take $\rho \equiv 1$. That means we are looking for (weak) solutions to the initial boundary value problem
\[ \left\{ \begin{array}{r c l l}
      \ddot{y}(x,t)-\divo (DW_{\rm cont}(\nabla y (x,t))) &=& f(x,t) &  \text{in}\ \Omega \times (0,T), \\
      y(x,t) & = & g(x,t) & \text{on}\ \partial\Omega \times (0,T),\\
      y(x,0) & = & h_0(x) & \text{in}\ \Omega,\\
      \dot{y}(x,0) & = & h_1(x) & \text{in}\ \Omega.      
\end{array} \right. \]

The assumptions on $W_{\rm cont}$ should be weak enough to be consistent with typical interatomic interaction potentials, e.g., Lennard-Jones potentials. Therefore, we should not assume global (quasi-)convexity or growth at infinity and $W_{\rm cont}$ should be allowed to have singularities. Still one can solve the problem quite generally, as long as the energy density and all the data are sufficiently smooth and as long as the energy density is well-behaved at the initial datum.

\subsection{The Atomistic Model}

We will mostly use the same notation as in \cite{braun16static}. We consider the reference lattice $\eps \Z^d$, where $\eps > 0$ is the lattice spacing. This partitions $\R^d$ into the cubes $\{z\} + \big(-\frac{\eps}{2},\frac{\eps}{2}\big]^d$ with $z \in \eps \Z^d$. Given $x \in \R^d$, we then define $\hat{x}\in \eps \Z^d$ to be the midpoint of the corresponding cube and $Q_\eps(x)$ the cube itself. 

The actual position of the atoms are described by a deformation map $y \colon (\Omega \cap \eps \Z^d) \times [0,T) \to \R^d$. We want to look at a general finite range interaction model, i.e., there is a finite set $\mathcal{R} \subset \Z^d \backslash \{0\}$ denoting the possible interactions. We will always assume that $\spano_{\Z} \mathcal{R}=\Z^d$ and $\mathcal{R} = - \mathcal{R}$. In the energy, the atom marked by $x,\tilde{x} \in \eps \Z^d$ then can only interact directly if there is a $z\in \eps \Z^d$ with $x,\tilde{x} \in z + \eps \mathcal{R}$. Furthermore, we assume our system to be translationally invariant such that the interaction can only depend on the matrix of differences $D_{\mathcal{R},\eps} y (x) = (\frac{y(x+\eps \rho)-y(x)}{\eps})_{\rho \in \mathcal{R}}$ with $x \in \eps \Z^d$, where we already use the natural scaling that has an optimal interatomic distance on scale $\eps$. The site potential $W_{\rm atom} \colon (\R^d)^\mathcal{R} \to (-\infty,\infty]$ is then assumed to be independent of $\eps$. Compare \cite{BLL:02} for a detailed discussion of this scaling.

As a mild symmetry assumption on $W_{\rm atom}$, we will assume throughout that
\[ W_{\rm atom} (A) = W_{\rm atom} (T(A))\]
for all $A \in (\R^d)^\mathcal{R}$, where
\[T(A)_{\rho} = -A_{-\rho}.\]
This is indeed a quite weak assumption. In a typical situation this just means that we have partitioned the overall energy in such a way, that the site potential is invariant under a point reflection at that atom combined with the natural relabeling.

In particular, if $W_{\rm atom}$ is sufficiently smooth, we have
\[D^k W_{\rm atom} ((B \rho)_{\rho \in \mathcal{R}})[T(A_1), \dotsc, T(A_k)] = D^k W_{\rm atom} ((B \rho)_{\rho \in \mathcal{R}})[A_1, \dotsc, A_k].\]
Letting $R_{\rm max} = \max \{\lvert \rho \rvert \colon \rho \in \mathcal{R}\}$ and $R_0 = \max\{R_{\rm max}, \frac{\sqrt{d}}{4}\}$, the discrete gradient $D_{\mathcal{R},\eps}y$ is surely well-defined on the discrete 'semi-interior'
\[\sinto_\eps \Omega = \{x \in \Omega \cap \eps \Z^d \colon \dist(x,\partial \Omega) > \eps R_0\}.\]
Additionally, the definition of $R_0$ implies that
\[\Omega_\eps = \bigcup_{z \in \into_\eps \Omega} Q_\eps(z) \subset \Omega\]
which will be used later on.
The energy is then defined by a sum over $\sinto_\eps \Omega$, such that any variations should only change these gradients. Hence, variations should only be allowed on the full discrete interior
\[\into_\eps \Omega = \{x \in \Omega \cap \eps \Z^d \colon \dist(x,\partial \Omega) > 2\eps R_0\}\]
and the boundary values should be prescribed on the boundary layer $\partial_\eps \Omega = \Omega \cap \eps \Z^d \backslash \into_\eps \Omega$. Indeed, we consider a boundary datum $g_{\rm atom} \colon \partial_\eps \Omega \times [0,T) \to \R^d$ and define the set of admissible deformations for a fixed time as
\[ \mathcal{A}_\eps (\Omega, g) = \{ y \colon \Omega \cap \eps \Z^d \to \R^d \colon y(x)=g(x) \text{ for all } x \in \partial_\eps \Omega\}.\]
Given a body force $f_{\rm atom} \colon \into_\eps \Omega \times [0,T) \to \R^d$ we will look at the potential energy
\[
E_\eps(y;f_{\rm atom})(t) = \eps^d \Big( \sum\limits_{x \in \sinto_\eps \Omega} W_{\rm atom}(D_{\mathcal{R},\eps} y (x,t)) -\sum\limits_{x \in \eps \Z^d \cap \Omega} y(x,t)f_{\rm atom}(x,t) \Big).
\]
The scaling ensures that affine deformations have an energy independent of $\eps$ (up to lower order terms) and, more generally, sufficiently smooth deformations have a finite and non-trivial energy in the limit, cf.\ \cite{BLL:02}.

For later use, we define the atomistic (semi-)norms
\begin{align*}
\lVert y \rVert_{\ell^2_\eps(\into_\eps \Omega}^2 &= \eps^d \sum_{x \in \into_\eps \Omega} \lvert y(x) \rvert^2 \\ 
\lVert y \rVert_{h^1_\eps(\sinto_\eps \Omega)}^2 &= \eps^d \sum_{x \in \sinto_\eps \Omega} \lvert D_{\mathcal{R},\eps} y(x) \rvert^2.
\end{align*}
Note in particular the norm equivalency
\[\sup\limits_{x \in \sinto_\eps \Omega} \lvert D_{\mathcal{R},\eps} y(x) \rvert \leq \eps^{-\frac{d}{2}} \lVert y \rVert_{h^1_\eps(\sinto_\eps \Omega)} \]
which will play a crucial role later on.

Given $g \colon \partial_\eps \Omega \to \R^d$, $y \colon \Omega \cap \eps \Z^d$ minimizes $\lVert y \rVert_{h^1_\eps(\sinto_\eps \Omega)}$ under the constraint $y(x)=g(x)$ for all $x \in \partial_\eps \Omega$ if and only if $(y,u)_{h^1_\eps(\sinto_\eps \Omega)}=0$ for all $u \in \mathcal{A}_\eps(\Omega,0)$ and $y(x)=g(x)$ for all $x \in \partial_\eps \Omega$. Thus, for every $g \colon \partial_\eps \Omega \to \R^d$ there is precisely one such $y$, it depends linearly on $g$ and is the unique solution to $\divo_{\mathcal{R},\eps} D_{\mathcal{R},\eps} y = 0$ with boundary values $g$. We write $y=T_\eps g$. Accordingly, we define the semi-norm
\[ \lVert g \rVert_{\partial_\eps \Omega, 0} = \lVert T_\eps g \rVert_{h^1_\eps(\sinto_\eps \Omega)}.\]
This norm on the boundary does not play such an important role in the dynamic case as in the static case, but we will later define the more important $\lVert g \rVert_{\partial_\eps \Omega, dyn}$ in the same spirit.

In the static case one is interested in finding local minimizers. In the dynamic case we additionally have an initial configuration $h_{0,{\rm atom}} \colon (\Omega \cap \eps \Z^d) \to \R^d$ and initial velocities $h_{1,{\rm atom}} \colon (\Omega \cap \eps \Z^d) \to \R^d$ such that the compatibility conditions $h_{0,{\rm atom}} \in \mathcal{A}_\eps (\Omega, g_{\rm atom}(\cdot,0))$ and $h_{1,{\rm atom}} \in \mathcal{A}_\eps (\Omega, \dot{g}_{\rm atom}(\cdot,0))$ hold true. At last, let us assume that all atoms have the same mass $m_\eps = \eps^d \rho$. This scaling ensures that the macroscopic reference body has a finite positive mass density $\rho$. As remarked before, we can assume $\rho=1$. Note that the scaling of the potential energy and the masses only affects the scaling of time. With our choice the time will not be rescaled and remains a macroscopic quantity. For the body forces, this scaling corresponds to a macroscopic acceleration of each atom (e.g. through gravity). 

Again we apply Newton's second law of motion and arrive at the initial boundary value problem
\[ \left\{ \begin{array}{r c l l}
      \ddot{y}(x,t) -\divo_{\mathcal{R},\eps} \big( DW_{\rm atom}(D_{\mathcal{R},\eps} y (x,t))\big) &=& f_{\rm atom}(x,t) &  \text{in}\ \into_\eps \Omega \times (0,T), \\
      y(x,t) & = & g_{\rm atom}(x,t) & \text{on}\ \partial\Omega_\eps \times [0,T),\\
      y(x,0) & = & h_{0,{\rm atom}}(x) & \text{in}\ \Omega \cap \eps \Z^d,\\
      \dot{y}(x,0) & = & h_{1,{\rm atom}}(x) & \text{in}\ \Omega \cap \eps \Z^d      
\end{array} \right. \]
where $DW_{\rm atom}(M) = \big( \frac{\partial W_{\rm atom}(M)}{\partial M_{i\rho}} \big)_{\substack{1 \leq i \leq d \\ \rho \in \mathcal{R}}}$ for $M = (M_{i\rho})_{\substack{1 \leq i \leq d \\ \rho \in \mathcal{R}}} \in \R^{d \times \mathcal{R}} \cong (\R^d)^\mathcal{R}$ and we write
\[\divo_{\mathcal{R},\eps} M(x) = \sum\limits_{\rho \in \mathcal{R}} \frac{M_\rho (x) - M_\rho(x-\eps \rho)}{\eps}\]
for any $M \colon \Omega \cap \eps \Z^d \to \R^{d \times \mathcal{R}} \cong (\R^d)^\mathcal{R}$. There are, of course, no actual derivatives in space involved here. These are just our short notations for the finite difference operators.

\subsection{The Cauchy-Born Rule}
As described in detail in the introduction, it is a fundamental problem to identify the correct $W_{\rm cont}$ that should be taken for the continuous equation so that one can hope for atomistic solutions close by as $\eps$ becomes small enough. The classical ansatz to resolve this question by applying the Cauchy-Born rule, leads to setting $W_{\rm cont} = W_{\rm CB}$, where in our setting the Cauchy-Born energy density has the simple mathematical expression  
\[W_{\rm CB}(A) := W_{\rm atom} ((A\rho)_{\rho \in \mathcal{R}}).\]
In the following we will only consider $W_{\rm cont} = W_{\rm CB}$, where $W_{\rm atom}$ is given. One of our main goals is to justify this choice rigorously.

\section{Preparations} \label{sec:preparations}

\subsection{Stability}

First we want to discuss the question of stability. We will only give a short summary. All proofs can be found in \cite{braun16static} and even more details in \cite{braunphdthesis}.

For the continuous equations the correct notion of stability for a $A\in\R^{d \times d}$ is given by the positivity of the Legendre-Hadamard stability constant, which for any nonempty, open, bounded $U \subset \R^n$ is given by
\begin{align*}
\lambda_{\rm LH}(A) &:= \inf\limits_{u \in H^1_0(U;\R^d) \backslash \{0\}} \frac{\int_U D^2W_{\rm CB}(A)[\nabla u (x), \nabla u(x)] \,dx}{\int_U \lvert \nabla u (x) \rvert^2 \,dx}\\
&= \inf\limits_{\xi, \eta \in \R^d\backslash \{0\}} \frac{D^2W_{\rm CB}(A)[\xi \otimes \eta,\xi \otimes \eta]}{\lvert \xi \rvert^2 \lvert \eta \rvert^2}
\end{align*}

While this condition is the correct notion to ensure existence and uniqueness of solutions to the continuous equation, it turns out that it is to weak to guarantee that there is a solution to the atomistic problem close by.

For fixed $\eps>0$, a tensor $K \in \R^{(d \times \mathcal{R})\times (d \times \mathcal{R})}$, and a bounded, open, and non-empty $\Omega \subset \R^d$ we set
\[ \lambda_\eps(K, \Omega) = \inf\limits_{\substack{
            y \in \mathcal{A}_\eps(\Omega,0)\\
            y \neq 0}}
            \frac{ \eps^d \sum\limits_{x \in \sinto_\eps \Omega} K [D_{\mathcal{R},\eps} y (x),D_{\mathcal{R},\eps} y (x)]}{\eps^d \sum\limits_{x \in \sinto_\eps \Omega} \lvert D_{\mathcal{R},\eps} y (x)\rvert^2}.\]
We then define the atomistic stability constant by
\[\lambda_{\rm atom}(K)=\lim_{\eps \to 0}\lambda_\eps (K, \Omega) = \inf\limits_{\eps > 0} \lambda_\eps (K, \Omega).\]
In particular, we are interested in the cases $K=D^2W_{\rm atom}(A)$ for $A \in \R^{d \times \mathcal{R}}$ and $K=D^2W_{\rm atom}((A\rho)_{\rho \in \mathcal{R}})$ for $A \in \R^{d \times d}$. In either case we will just write $\lambda_{\rm atom}(A)$.

The limit in the definition exists, converges to the infimum as claimed, and is independent of $\Omega$, cf.\ \cite[Prop. 3.2]{braun16static}. Furthermore, one can show that looking at larger and larger boxes with periodic boundary conditions gives the same quantity, cf.\ \cite[Prop. 3.1]{braun16static}. In this way the notion of stability here can be shown to be equivalent to the notion in \cite{hudsonortner}. The only difference is a choice of equivalent norms.

One can also give a characterization in spirit of the Legendre-Hadamard condition
\begin{align*}
\lambda_{\rm atom}(K)= &\inf \bigg\{\frac{K[\xi \otimes c(k),\xi \otimes c(k)]}{\lvert \xi \rvert^2 (\lvert c(k) \rvert^2 + \lvert s(k) \rvert^2)}\\ &+ \frac{K[\xi \otimes s(k), \xi \otimes s(k)]}{\lvert \xi \rvert^2 (\lvert c(k) \rvert^2 + \lvert s(k) \rvert^2)} \colon\\
&\qquad \xi \in \R^d \backslash \{0\}, k \in [0,2\pi)^d\backslash \{0\} \bigg\},
\end{align*}
where $c(k)_\rho = \cos(\rho k) -1$ and $s(k)_\rho = \sin(\rho k)$, cf.\ \cite[Cor. 3.7]{braun16static}. In the case $K=D^2W_{\rm atom}((A\rho)_{\rho \in \mathcal{R}})$ it is very easy to see, by looking at the liminf as $k \to 0$ instead of the full infimum, that $\lambda_{\rm atom}(A) \leq C \lambda_{\rm LH}(A)$ for $A \in \R^{d \times d}$. The constant $C$, again, is just a consequence of the choice of equivalent norms. Since here $k$ is indeed a wave number, this is a quite intuitive property: For a crystalline material to be stable, it has to be stable by perturbations on all wave lengths. In contrast, for the continuous equations the stability under long wave length perturbations is sufficient, which corresponds to the long wave length limit $k \to 0$.

We also note, that the stability constants $\lambda_{\rm atom}(A)$ and $\lambda_{\rm LH}(A)$ depend continuously on $A$ as long as $W_{\rm atom} \in C^2$.

In \cite{braun16static} we also give criteria for atomistic stability that are simpler to check but not as sharp. Additionally, we discuss examples analytically and, in particular, give an example that $\lambda_{\rm atom}(A) < 0 < \lambda_{\rm LH}(A)$ can indeed occur.

\subsection{Residual Estimates}

Here we just want to state the crucial residual estimate as well as two results on approximations. These results have been proven in \cite{braun16static}. To avoid stronger regularity assumptions, it is important to not just estimate the residuum at the atom sites by using the continuum equations there. Instead, one uses the continuum equations at every point and gets rid of certain error terms by averaging. Additionally, the norms in the error terms can be improved with a regularization of the continuum solution.

\begin{prop} \label{prop:ell2residuum}
Let $V \subset \R^{d \times \mathcal{R}}$ be open and $W_{\rm atom} \in C^4_b(V)$. Let $f \in L^2(\Omega;\R^d)$ and set
\[ \tilde{f}(x) = \fint_{Q_\eps(x)} f(a)\,da\]
for $x \in \into_\eps \Omega$. Furthermore let $\eps \in (0,1]$ and $y \in C^{3,1}(\R^d;\R^d)$ with
\begin{align*}
\co \{D_{\mathcal{R},\eps} y (\hat{x}+ \eps \sigma), (\nabla y (x) \rho)_{\rho \in \mathcal{R}}\} \subset V
\end{align*}
for all $x \in \Omega_\eps$ and $\sigma \in \mathcal{R} \cup \{0\}$. Then we have
\begin{align*}
\big\lVert -\tilde{f} &- \divo_{\mathcal{R},\eps} \big( DW_{\rm atom}(D_{\mathcal{R},\eps} y)\big) \big\rVert_{\ell_\eps^2(\into_\eps \Omega)}\\
&\leq \lVert -f - \divo DW_{\rm CB}(\nabla y)\rVert_{L^2(\Omega_\eps; \R^d)}+ C \eps^2 \Big\lVert \lVert \nabla^4 y \rVert_{L^\infty(B_{\eps R}(x))}\\
&+ \lVert \nabla^3 y \rVert_{L^\infty(B_{\eps R}(x))}^\frac{3}{2} + \lVert \nabla^2 y \rVert_{L^\infty(B_{\eps R}(x))}^3 + \eps\lVert \nabla^3 y \rVert_{L^\infty(B_{\eps R}(x))}^2 \Big\rVert_{L^2(\Omega_\eps)},
\end{align*}
where $\Omega_\eps = \bigcup_{z \in \into_\eps \Omega} Q_\eps(z)$, $R=2R_{\rm max}+\frac{3\sqrt{d}}{2}$ and $C = C(d,\mathcal{R}, \lVert D^2 W_{\rm atom} \rVert_{C^2(V)}) >0$.
\end{prop}

These residual estimates are particularly strong if we combine them with the following two approximation results:
\begin{prop} \label{prop:approximation1}
For any $R>0$, $k,d \in \N$, $p\geq 1$, there is a $C=C(R,d,p)>0$ such that for any $U \subset \R^d$ measurable and $y \in W^{k,p}(U + B_{(R+1)\eps}(0);\R^d)$ we have
\begin{align*}
\Big\lVert \lVert \nabla^k (y \ast \eta_\eps) \rVert_{L^\infty(B_{\eps R}(\cdot))} \Big\rVert_{L^p(U)} \leq C \lVert \nabla^k y \rVert_{L^{p}(U+B_{(R+1)\eps}(0))},
\end{align*}
where $\eta_\eps$ is the standard scaled smoothing kernel.
\end{prop}
\begin{prop} \label{prop:approximation2}
Let $d \in \{1,2,3,4\}$, $\Omega \subset \R^d$ open and bounded with Lipschitz boundary, $V \subset \R^{d \times \mathcal{R}}$ be open and $W_{\rm atom} \in C^5_b(V)$. Then, there is a $C>0$ such that for all $\eps \in (0,1]$ and all $y \in H^4(\Omega + B_{\eps}(0);\R^d)$
with
\begin{align*}
\inf_{x\in\Omega}\inf_{t \in [0,1]} \dist ((1-t)(\nabla y (x) \rho)_{\rho \in \mathcal{R}}+t(\nabla (y \ast \eta_\eps) (x) \rho)_{\rho \in \mathcal{R}}, V^c)>0,
\end{align*}
we have
\begin{align*}
\lVert \divo & DW_{\rm CB}(\nabla y(x)) - \divo DW_{\rm CB}(\nabla (y \ast \eta_\eps)(x)) \rVert_{L^2(\Omega)}\\
& \leq C \eps^2 \big( \lVert \nabla^2 y \rVert_{L^4(\Omega+B_{\eps}(0))} \lVert \nabla^3 y \rVert_{L^4(\Omega+B_{\eps}(0))} + \lVert \nabla^4 y \rVert_{L^2(\Omega+B_{\eps}(0))} \big)
\end{align*}
where $\eta_\eps$ is the standard scaled smoothing kernel.
\end{prop}

\section{Continuum Elastodynamics} \label{sec:cont}

Before we can discuss the atomistic equations, we have to discuss the continuous Cauchy-Born problem. We are not only interested in existence and uniqueness, but also in the maximal existence interval and higher order regularity. The most important part of the result is the short time existence, already contained in \cite{dafermoshrusa85}. But, since we want to discuss the atomistic equations for long times, it is important that extend this short time result to a result about the maximal existence time. This will still require a considerable amount of work. Furthermore, there are two key regularity theorems that are only stated and used in \cite{dafermoshrusa85}. Their proofs were left out by the authors. Since we will also need these statements in our proof directly, we will prove them in the appendix. Theorem \ref{thm:additionalregularity} is about additional regularity for solutions of second order hyperbolic equations, while Theorem \ref{thm:optimalsobolevregularity} is about higher order elliptic regularity under very weak assumptions on the coefficients.

If we want to achieve higher regularity for such a second order hyperbolic initial-boundary-value problem, compatibility conditions on $f,g,h_0,h_1$ are crucial. We say that $f,g,h_0,h_1$ satisfy the compatibility conditions of order $m$, if 
\[ u_k := \frac{\partial^k}{\partial t^k} (y-g) |_{t=0} \in H^1_0(\Omega;\R^d)\]
for all $k \in \{0,\dotsc, m-1\}$ as computed formally using the equation in terms of $f,g,h_0,h_1$. This can be written explicitly and, even though the expressions are quite nasty, we still want to do so in order to be able to discuss some regularity issues in more detail.

If $m-1 > \frac{d}{2}$, $h_0 \in H^m(\Omega; \R^d)$, $h_1 \in H^{m-1}(\Omega; \R^d)$, $\partial_t^{k-2} f(\cdot,0) \in H^{m-k}(\Omega; \R^d)$ for $2 \leq k \leq m-1$, $\partial_t^{k} g(\cdot,0) \in H^{m-k}(\Omega; \R^d)$ for $0 \leq k \leq m-1$ and $W_{\rm CB} \in C^{2m-2}$ on an open set containing $\{\nabla h_0(x) \colon x \in \Omega\}$, then we define $u_0(x) = h_0(x) -g(x,0)$, $u_1(x) = h_1(x) - \dot{g}(x,0)$,
\[u_2(x) = f(x,0) - \ddot{g}(x,0)+ \divo \big( DW_{\rm CB}(\nabla h_0(x)) \big)\]
and recursively, for $3 \leq k \leq m-1$,
\begin{align*}
(u_k(x)&)_i = \partial_t^{k-2} f_i(x,0) - \partial_t^{k} g_i(x,0)\\
&+\sum_{j,q,r=1}^d D^{E_{ij}+E_{qr}}W_{\rm CB}(\nabla h_0(x)) \Big( \partial_{x_r} \partial_{x_j} \partial_t^{k-2} g_q(x,0) + \partial_{x_r} \partial_{x_j} (u_{k-2})_q(x)\Big) \\
&+\sum_{j,q,r=1}^d \sum_{n=1}^{k-2} \sum_{\substack{
            \beta \in \N_0^{d \times d}\\
            1 \leq \vert \beta \rvert \leq n}} \sum_{s=1}^n \sum_{p_s(n,\beta)} \frac{(k-2)!}{(k-2-n)!} D^{\beta+E_{ij}+E_{qr}}W_{\rm CB}(\nabla h_0(x)) \\
&\quad \cdot \Big( \prod_{l=1}^s \frac{(\partial_t^{\gamma_l} \nabla g(x,0) + \nabla u_{\gamma_l}(x))^{\lambda_l}}{\lambda_l! (\gamma_l!)^{\lambda_l}} \Big) \Big( \partial_{x_r} \partial_{x_l} \partial_t^{k-2-n} g_q(x,0) + \partial_{x_r} \partial_{x_l} (u_{k-2-n})_q(x)\Big),
\end{align*}
where
\begin{align*}
p_s(n,\beta)&=\Big\{ (\lambda_1,\dots,\lambda_s;\gamma_1,\dots,\gamma_s) \colon \lambda_l \in \N_0^{d \times d}, \gamma_l \in \N_0,\\
&\quad 0< \gamma_1 < \dots < \gamma_s, \lvert \lambda_l \rvert > 0, \sum_{l=1}^s \lambda_l = \beta, \sum_{l=1}^s \gamma_l \lvert \lambda_l \rvert = n  \Big\},
\end{align*}
and $E_{ij}$ is the matrix with $(E_{ij})_{ij}=1$ and zeros everywhere else.
The following result shows that $u_k \in H^1_0(\Omega;\R^d)$ is only a condition on the boundary values and not on the regularity.

\begin{prop} \label{prop:compatibilitycond1}
If $m-1 > \frac{d}{2}$, $h_0 \in H^m(\Omega; \R^d)$, $h_1 \in H^{m-1}(\Omega; \R^d)$, $\partial_t^{k-2} f(\cdot,0) \in H^{m-k}(\Omega; \R^d)$ for $2 \leq k \leq m-1$, $\partial_t^{k} g(\cdot,0) \in H^{m-k}(\Omega; \R^d)$ for $0 \leq k \leq m-1$ and $W_{\rm CB} \in C^{m}$ on an open set containing $\{\nabla h_0(x) \colon x \in \Omega\}$, then $u_k \in H^{m-k}(\Omega; \R^d)$ for all $0 \leq k \leq m-1$.
\end{prop}
\begin{proof}
This is clear for $k=0,1$. For $k=2$ this follows directly from the arguments in Lemma \ref{lem:composition}. If $k \geq 3$, inductively, the same arguments show $D^{\beta+E_{ij}+E_{qr}}W_{\rm CB} \circ \nabla h_0 \in H^{m-1}$ and then one can apply Lemma \ref{lem:prodofsobolev} with $M = m-1$ and $N=0 + \sum \gamma_j \lvert \lambda_j \rvert +(k-l-1) = k-1$ to estimate the product and gives the desired result. Actually, for this to be completely true, we would need the stronger assumption $W_{\rm CB} \in C^{2m-2}$ so that $D^{\beta + E_{ij}+ E_{qr}}W_{\rm CB} \circ \nabla h_0 \in H^{m-1}$. To reduce this assumption to $W_{\rm CB} \in C^{m}$, we note that in the application of Lemma \ref{lem:prodofsobolev} we only take the $\alpha$-th derivative of $v=D^{\beta+E_{ij}+E_{qr}}W_{\rm CB} \circ \nabla h_0$ with $0 \leq \lvert \alpha \rvert \leq m-k =M-N$ and then have to know that $D^\alpha v \in L^q$ for a certain $q$ formerly coming from the Sobolev embedding of $H^{m-1-\lvert \alpha \rvert}$. Now we have to prove this estimate differently.
From Corollary \ref{cor:FaadiBruno} we know that
\[ \lvert D^\alpha v(x) \rvert \leq C\sum_{r=1}^{\lvert \alpha \rvert} \lvert D^{r+2+ \lvert \beta \rvert}W_{\rm CB}(\nabla h_0(x)) \rvert \sum_{\substack{l_1, \dots, l_r\geq 1 \\ l_1+ \dots+ l_r = \lvert \alpha \rvert}} \prod_{j=1}^r \lvert D^{(l_j+1)}h_0(x)\rvert,\]
if $W_{\rm CB} \in C^m$ and $h_0 \in C^{1+\lvert \alpha \rvert}$. But of course this extends to $h_0 \in H^m$ once we estimate the product on the right hand side suitably. These estimates, which also give the desired integrability of $D^\alpha v$ follow along the lines of the proof of Lemma \ref{lem:prodofsobolev}.
\end{proof}

If we already have a solution and use it as a starting point, then the compatibility conditions are automatically satisfied and the $u_k$ are indeed directly given by $y-g$.
\begin{prop} \label{prop:compatibilitycond2}
Let $m\in\N$ with $m > \frac{d}{2} + 1$, $\delta>0$, let $\Omega \subset \R^d$ be an open, bounded set with $\partial \Omega$ of class $C^{m}$, $V \subset \R^{d \times d}$ open, $W_{\rm CB} \in C_b^{m+1}(V)$,
\begin{align*}
&f \in C^{m-1}(\overline{\Omega}\times [-\delta,\delta]; \R^d),\\
&g \in C^{m+1}(\overline{\Omega}\times [-\delta,\delta]; \R^d) \text{, and}\\
&y \in \bigcap_{k=0}^m C^k\big([-\delta,\delta]; H^{m-k}(\Omega;\R^d)\big) \ \text{with}\\
&\overline{\{\nabla y(x,t) \colon  x \in \Omega, t \in [-\delta,\delta]\}} \subset V_{\rm LH}.
\end{align*}
Furthermore let $y$ be a solution of the equations. If we now set $h_0 = y(0)$ and $h_1=\partial_t y(0)$, then we have
\[ u_k = \frac{\partial^k}{\partial t^k} (y-g) |_{t=0} \in H^1_0(\Omega;\R^d)\]
for all $k \in \{0, \dotsc, m-1\}$.
\end{prop}
\begin{proof}
Since $y-g \in C^{m-1}\big([-\delta,\delta]; H^1(\Omega;\R^d)\big)$, $H^1_0(\Omega;\R^d)$ is a closed subspace of $H^1(\Omega;\R^d)$ and $y(t)-g(t) \in H^1_0(\Omega;\R^d)$ for all $t \in [- \delta, \delta]$, we clearly find
\[\frac{\partial^k}{\partial t^k} (y-g) |_{t=0} \in H^1_0(\Omega;\R^d).\]

Let us now proof $u_k = \frac{\partial^k}{\partial t^k} (y-g) |_{t=0}$ by induction over $k$. By definition this is true for $k=0,1$. $k=2$ follows from the equation. If now $3 \leq k \leq m-1$, we have to show that the recursion formula for the $u_k$ also holds for the derivatives of $y-g$. Clearly we have
\begin{align*}
\partial_t^k (y-g) &=  \partial_t^{k-2} \Big( f - \partial_t^2 g + \divo (DW_{\rm CB}(\nabla y))\Big)\\
&= \partial_t^{k-2} f - \partial_t^k g + \partial_t^{k-2} \divo (DW_{\rm CB}(\nabla y)).
\end{align*}
If $y$ were smooth the last term can be written explicitly with the chain rule, the Leibniz rule, as well as the generalized Faà di Bruno formula, Lemma \ref{lem:multivariateFaaDiBruno}. We first get
\[\divo (DW_{\rm CB}(\nabla y))_i = \sum_{j,q,r=1}^d D^{E_{ij}+E_{qr}}W_{\rm CB}(\nabla y) \partial_{x_r} \partial_{x_j} y_q, \]
then
\begin{align*}
\partial_t^{k-2} \divo (DW_{\rm CB}(\nabla y))_i = \sum_{j,q,r=1}^d \sum_{n=0}^{k-2} \binom{k-2}{n} \partial_t^n(D^{E_{ij}+E_{qr}}W_{\rm CB}(\nabla y)) \partial_t^{k-2-n} \partial_{x_r} \partial_{x_j} y_q,
\end{align*}
and finally
\begin{align*}
(\partial_t^k (y-g))_i &= \partial_t^{k-2} f_i - \partial_t^{k} g_i\\
&+\sum_{j,q,r=1}^d D^{E_{ij}+E_{qr}}W_{\rm CB}(\nabla y) \partial_{x_r} \partial_{x_j} \partial_t^{k-2} y_q \\
&+\sum_{j,q,r=1}^d \sum_{n=1}^{k-2} \sum_{\substack{
            \beta \in \N_0^{d \times d}\\
            1 \leq \vert \beta \rvert \leq n}} \sum_{s=1}^n \sum_{p_s(n,\beta)} \frac{(k-2)!}{(k-2-n)!} D^{\beta+E_{ij}+E_{qr}}W_{\rm CB}(\nabla y) \\
            &\quad \cdot \Big( \prod_{l=1}^s \frac{(\partial_t^{\gamma_l} \nabla y)^{\lambda_l}}{\lambda_l! (\gamma_l!)^{\lvert \lambda_l \rvert}} \Big) \partial_{x_r} \partial_{x_l} \partial_t^{k-2-n} y_q,
\end{align*}
where
\begin{align*}
p_s(n,\beta)&=\Big\{ (\lambda_1,\dots,\lambda_s;\gamma_1,\dots,\gamma_s) \colon \lambda_l \in \N_0^{d \times d}, \gamma_l \in \N_0,\\
&\quad 0< \gamma_1 < \dots < \gamma_s, \lvert \lambda_l \rvert > 0, \sum_{l=1}^s \lambda_l = \beta, \sum_{l=1}^s \gamma_l \lvert \lambda_l \rvert = n  \Big\}.
\end{align*}
Due to the bounds discussed in Proposition \ref{prop:compatibilitycond1} and Lemma \ref{lem:prodofsobolev} this still holds under the given weaker regularity assumption on $y$. Inductively, we thus have proven the claim.
\end{proof}

In the following, for $V \subset \R^{d \times d}$ open and $W_{\rm CB} \in C^2(V)$ we write
\[V_{\rm LH} = \{A \in V \colon \lambda_{\rm LH}(A)>0\},\]
which is again an open set, since $\lambda_{\rm LH}$ is continuous.

Let us start with a local existence result.
\begin{thm} \label{thm:localcontwave}
Let $m\in\N$ with $m > \frac{d}{2} +2$, $T_0>0$, let $\Omega \subset \R^d$ be an open, bounded set with $\partial \Omega$ of class $C^{m}$, $V \subset \R^{d \times d}$ open and $W_{\rm CB} \in C_b^{m+1}(V)$. Given a body force $f$, initial data $h_0,h_1$ and boundary values $g$ such that
\begin{align*}
&f \in C^{m-1}(\overline{\Omega}\times [0,T_0]; \R^d)\\
&g \in C^{m+1}(\overline{\Omega}\times [0,T_0]; \R^d)\\
&h_0 \in H^m (\Omega; \R^d)\\
&\overline{\{\nabla h_0(x) \colon  x \in \Omega\}} \subset V_{\rm LH} \\
&h_1 \in H^{m-1} (\Omega; \R^d)
\end{align*}
and such that the compatibility conditions of order $m$ are satisfied (see above).

Then, for all sufficiently small $T \in (0,T_0]$ the problem has a unique solution
\[ y \in \bigcap_{k=0}^m C^k\big([0,T]; H^{m-k}(\Omega;\R^d)\big)\]
and we have
\[\overline{\{\nabla y(x,t) \colon  x \in \Omega, t \in [0,T]\}} \subset V_{\rm LH}.\]
\end{thm}
\begin{proof}
This follows from \cite[Thm.~5.1]{dafermoshrusa85} by setting $\textbf{u}^0 = h_0-g(\cdot,0)$, $\textbf{u}^1 = h_1-\dot{g}(\cdot,0)$,
\[ \textbf{g}(x,t,u,p,M)_k = f(x,t)_k - \ddot{g}(x,t)_k + \sum \frac{\partial^2 W_{\rm CB}}{\partial a_{ki} \partial a_{lj}}(M+\nabla g(x,t)) \frac{\partial^2 g_l}{\partial x_i \partial x_j}(x,t)\]
and
\[ (\textbf{A}_{ij})_{kl} (x,t,u,p,M) = \chi (M+\nabla g(x,t))\frac{\partial^2 W_{\rm CB}}{\partial a_{ki} \partial a_{lj}}(M+\nabla g(x,t)) + (1-\chi(M+\nabla g(x,t))) \delta_{kl} \delta_{ij},\]
where $\chi \colon \R^{d \times d} \to [0,1]$ is a smooth cutoff with $\chi(M)=1$ for $M \in W_1$ and $\chi(M)=0$ for $M \notin W_2$ and $W_1,W_2$ are open sets such that
\[\overline{\{\nabla h_0(x) \colon  x \in \Omega\}} \css W_1 \css W_2 \css V_{\rm LH}.\]
We then set $y = u+g$ and reduce the existence time $T$ enough to ensure $\nabla y(x,t) \in W_1$ for all $(x,t) \in \overline{\Omega} \times [0,T]$.
\end{proof}

We can use this local result to construct a maximal solution.
\begin{thm} \label{thm:maxcontwave}
Let $m\in\N$ with $m > \frac{d}{2} +2$, $T_0>0$, let $\Omega \subset \R^d$ be an open, bounded set with $\partial \Omega$ of class $C^{m}$, $V \subset \R^{d \times d}$ open and $W_{\rm CB} \in C^{m+1}(V)$. Given a body force $f$, initial data $h_0,h_1$ and boundary values $g$ such that
\begin{align*}
&f \in C^{m-1}(\overline{\Omega}\times [0,T_0]; \R^d)\\
&g \in C^{m+1}(\overline{\Omega}\times [0,T_0]; \R^d)\\
&h_0 \in H^m (\Omega; \R^d)\\
&\overline{\{\nabla h_0(x) \colon  x \in \Omega\}} \subset V_{\rm LH} \\
&h_1 \in H^{m-1} (\Omega; \R^d)
\end{align*}
and such that the compatibility conditions of order $m$ are satisfied.

Then there are unique $T_{\rm cont} > 0$ and
\[ y \in \bigcap_{k=0}^m C^k\big([0,T_{\rm cont}); H^{m-k}(\Omega;\R^d)\big),\]
such that
\[\{\nabla y(x,t) \colon  x \in \overline{\Omega}, t \in [0,T_{\rm cont})\} \subset V_{\rm LH},\]
$y$ is a solution on $[0,T_{\rm cont})$ and at least one of the following conditions is true:
\begin{enumerate}[label=(\roman*)]
\item $T_{\rm cont}=T_0$,
\item $\liminf_{t \to T_{\rm cont}}\dist \big(V_{\rm LH}^c, \{\nabla y(x,t) \colon  x \in \overline{\Omega}\}\big) = 0$,
\item $\limsup_{t \to T_{\rm cont}} \lVert y(t) \rVert_{H^m(\Omega;\R^d)} + \lVert \partial_t y(t) \rVert_{H^{m-1}(\Omega;\R^d)} = \infty$.
\end{enumerate}
\end{thm}
\begin{proof}
Let $T_{\rm cont}$ be the supremum of all $0<T \leq T_0$ such that there is a
\[ y \in \bigcap_{k=0}^m C^k\big([0,T]; H^{m-k}(\Omega;\R^d)\big)\]
with
\[\overline{\{\nabla y(x,t) \colon  x \in \Omega, t \in [0,T]\}} \subset V_{\rm LH}\]
that solves the problem on $[0,T]$. Theorem \ref{thm:localcontwave} ensures that there is at least one such $T$. If we take any of these solutions and $t_0 \in (0,T)$ then $h_0=y(t_0)$, $h_1= \partial_t y(t_0)$ as well as the translated $f$ and $g$ can be used in Theorem \ref{thm:localcontwave} to show existence and uniqueness in some $[t_0, t_0+\delta]$, $\delta>0$. This is possible since the new $u_k$ satisfies $u_k =  \frac{\partial^k}{\partial t^k} (y-g) |_{t=t_0} \in H^1_0$ by Proposition \ref{prop:compatibilitycond2}.

The uniqueness ensures in particular, that all these $y$ are equal pairwise on the intersection of their existence intervals. Therefore, we have a
\[ y \in \bigcap_{k=0}^m C^k\big([0,T_{\rm cont}); H^{m-k}(\Omega;\R^d)\big)\]
with
\[\{\nabla y(x,t) \colon  x \in \overline{\Omega}, t \in [0,T_{\rm cont})\} \subset V_{\rm LH}\]
that solves the problem on $[0,T_{\rm cont})$.

Now assume $T_{\rm cont} < T_0$,
\[\liminf_{t \to T_{\rm cont}}\dist \big(V_{\rm LH}^c, \{\nabla y(x,t) \colon  x \in \overline{\Omega}\}\big)>0\]
and
\[\limsup_{t \to T_{\rm cont}} \lVert y(t) \rVert_{H^m(\Omega;\R^d)} + \lVert \partial_t y(t) \rVert_{H^{m-1}(\Omega;\R^d)} < \infty.\]
Then
\[y \in L^\infty(0,T_{\rm cont}; H^m(\Omega;\R^d)) \cap W^{1,\infty}(0,T_{\rm cont}; H^{m-1}(\Omega;\R^d)).\]
We claim that
\[\partial_t^k y \in L^\infty(0,T_{\rm cont}; H^{m-k}(\Omega;\R^d))\]
for $0 \leq k \leq m$. We already know this $k=0,1$. For $2 \leq k \leq m$ we represent the derivatives of $y$ as we did in Proposition \ref{prop:compatibilitycond2} and then argue inductively as in the proof of Proposition \ref{prop:compatibilitycond1}.

In particular, the limit $\tilde{h}_k := \lim_{t \to T_{\rm cont}} \partial_t^k y (t)$ exists strongly in $H^{m-k-1}(\Omega;\R^d)$ and weakly in $H^{m-k}(\Omega;\R^d)$ for $0 \leq k \leq m-1$. Since $H^{m-1}(\Omega;\R^d) \hookrightarrow C^1(\overline{\Omega};\R^d)$ we also have the convergence $y(t) \to \tilde{h}_0 $ in $C^1(\overline{\Omega};\R^d)$. In particular,
\[\overline{\{\nabla \tilde{h}_0(x) \colon  x \in \Omega\}} \subset V_{\rm LH}.\]
Now we want to use the local existence result, Theorem \ref{thm:localcontwave}, with shifted $f,g$ and initial conditions $\tilde{h}_0,\tilde{h}_1$. All we have to do, is to check that the compatibility conditions of order $m$ are satisfied. For $k=0$ or $k=1$, we clearly have
\[u_k = \tilde{h}_k - \partial_t^k g(\cdot,T_{\rm cont}) = \lim_{t \to T_{\rm cont}} \partial_t^k( y(\cdot,t) - g(\cdot,t)) \in H^1_0.\]
For $2 \leq k \leq m-1$, we know that $\partial_t^k(y-g)(t)$ converges to $\tilde{h}_k - \partial_t^k g(\cdot,T_{\rm cont})$ strongly in $H^{m-k-1}(\Omega;\R^d)$ and weakly in $H^{m-k}(\Omega;\R^d)$. Therefore, $\tilde{h}_k - \partial_t^k g(\cdot,T_{\rm cont})\in H^1_0(\Omega;\R^d)$. Now we just have to argue inductively that $u_k = \tilde{h}_k - \partial_t^k g(\cdot,T_{\rm cont})$. If this is already true for all $l<k$, we know in particular that $\partial_t^l(y-g)(t) \wto u_l$ in $H^{m-l}(\Omega;\R^d)$. Expressing $\partial_t^k(y-g)(t)$ with the equation in terms of $\partial_t^l (y-g)$, $0 \leq l \leq k-2$ as in Proposition \ref{prop:compatibilitycond2}, we can thus conclude that $\partial_t^k(y-g)(t) \to u_k$ at least in some weaker sense, e.g. in $H^{m-k-1}(\Omega;\R^d)$. To see this one needs to combine the arguments in Proposition \ref{prop:compatibilitycond1} with the statement on weak-to-strong continuity in Lemma \ref{lem:prodofsobolev} with $M=m-1$, $N=k-1$, $L=m-k-1$. Therefore, $f(\cdot, T_{\rm cont} + \cdot)$, $g(\cdot, T_{\rm cont} + \cdot)$, $\tilde{h}_0$, and $\tilde{h}_1$ satisfy the compatibility conditions of order $m$.

Hence, we can use Theorem \ref{thm:localcontwave} to find a $\delta>0$ and an extension of $y$ to $[0, T_{\rm cont}+\delta]$, such that
\[ y \in \bigcap_{k=0}^m C^k\big([T_{\rm cont},T_{\rm cont}+\delta]; H^{m-k}(\Omega;\R^d)\big),\]
$y$ is a solution of the equation on $(T_{\rm cont},T_{\rm cont}+\delta)$ with $y(T_{\rm cont})= \tilde{h}_0$ and $\dot{y}(T_{\rm cont})= \tilde{h}_1$. Here, $\dot{y}(T_{\rm cont})$ is to be understood in terms of the values on $[T_{\rm cont},T_{\rm cont}+\delta]$ alone. Furthermore, we have
\[\overline{\{\nabla y(x,t) \colon  x \in \Omega, t \in [T_{\rm cont},T_{\rm cont}+\delta]\}} \subset V_{\rm LH}.\]
We have to take a closer look at what happens in $T_{\rm cont}$. We clearly have 
\[u_k = \lim_{t \to T_{\rm cont}^+} \partial_t^k( y(\cdot,t) - g(\cdot,t))\]
strongly in $H^{m-k}(\Omega;\R^d)$. But we already saw that $u_k = \tilde{h}_k - \partial_t^k g(\cdot,T_{\rm cont})$ for $0 \leq k \leq m-1$. So the weak derivatives are continuous, which directly implies the strong differentiability
\[ y \in \bigcap_{k=0}^{m-1} C^k\big([0,T_{\rm cont}+\delta]; H^{m-k-1}(\Omega;\R^d)\big).\]
Furthermore, we have one more strong derivative outside of $T_{\rm cont}$ which extends to the entire interval including $T_{\rm cont}$ as a weak derivative. By continuity it is bounded on $[T_{\rm cont},T_{\rm cont} + \delta]$ and we have already shown the boundedness on $[0,T_{\rm cont})$. Therefore,
\[ y \in \bigcap_{k=0}^{m} W^{k,\infty}\big([0,T_{\rm cont}+\delta]; H^{m-k}(\Omega;\R^d)\big).\]
Additionally, by compactness and identification $\partial_t^k y$ is continuous in $H^{m-k}(\Omega;\R^d)$ with respect to the weak topology for all $0 \leq k \leq m-1$.

Now, we want to use the ideas of \cite{strauss66} to get the missing additional regularity. The key is to use that $y$ solves an equation.

Clearly $v:= \partial_t^{m-1} (y-g)$ satisfies $v \in L^\infty(0,T_{\rm cont}+\delta; H^1_0(\Omega;\R^d))$ with weak derivative $\partial_t v \in L^\infty(0,T_{\rm cont}+\delta; L^2(\Omega;\R^d))$. We claim that it also has a weak second derivative in $L^\infty(0,T_{\rm cont}+\delta; H^{-1}(\Omega;\R^d))$. To that end, we calculate
\begin{align*}
\partial_t^{m-1} (DW_{\rm CB} (\nabla y)_{ij}) &= D^2 W_{\rm CB} (\nabla y)[\nabla v, E_{ij}] + D^2 W_{\rm CB} (\nabla y)[\partial_t^{m-1} \nabla g, E_{ij}]\\
&+ \sum_{\substack{
            \beta \in \N_0^{d \times d}\\
            2 \leq \vert \beta \rvert \leq m-1}}
            \sum_{s=1}^{m-1} \sum_{p_s(m-1,\beta)} (m-1)! D^{\beta+E_{ij}}W_{\rm CB}(\nabla y) \prod_{j=1}^s \frac{(\partial_t^{\gamma_j}\nabla y)^{\lambda_j}}{\lambda_j! \gamma_j!^{\lvert \lambda_j \rvert}}
            \\
            &=: D^2 W_{\rm CB} (\nabla y)[\nabla v, E_{ij}]+ R_{ij}.
\end{align*}
We can now use Lemma \ref{lem:prodofsobolev} with $M=m-2$ and $N = \sum \lvert \lambda_j \rvert (\gamma_j -1) = m-1-\lvert \beta \rvert \leq m-3$ to see that
\[\prod_{j=1}^s (\partial_t^{\gamma_j}\nabla y)^{\lambda_j} \in L^\infty (0,T_{\rm cont}+\delta; H^1(\Omega;\R^d)).\]
Since
\[D^{\beta+E_{ij}}W_{\rm CB}(\nabla y) \in L^\infty (0,T_{\rm cont}+\delta; W^{1,\infty}(\Omega;\R^d)),\]
we obtain
\[R \in L^\infty (0,T_{\rm cont}+\delta; H^1(\Omega;\R^{d \times d}))\]
and
\[F:= \partial_t^{m-1}f - \partial_t^{m+1}g + \divo R \in L^\infty (0,T_{\rm cont}+\delta; L^2(\Omega;\R^d)).\]
Defining $A(t) : H^1_0(\Omega;\R^d) \to H^{-1}(\Omega;\R^d)$ by
\[A(t)u = -\divo(D^2 W_{\rm CB} (\nabla y(\cdot,t))[\nabla u]),\]
we can use a weak formulation (in time and space) of the equation to see that indeed $\partial_t^2 v$ exists as a weak derivative in $L^\infty (0,T_{\rm cont}+\delta; H^{-1}(\Omega;\R^d))$ and satisfies
\[\partial_t^2 v(t) + A(t)v(t) = F(t).\]
Let us look more precisely at $A$. Since $\nabla y \in C([0,T_{\rm cont}+\delta] \times \overline{\Omega}; \R^{d \times d})$, the coefficients $D^2 W_{\rm CB} (\nabla y(x,t))$ are uniformly bounded, uniformly continuous and have a positive, uniform Legendre-Hadamard constant. Therefore, it is well known that $A(t)$ satisfies a Gårding-inequality uniformly in time, see Theorem \ref{thm:contGårding}. I.e., there are $\lambda_1 >0$, $\lambda_2 \in \R$ such that
\[\langle A(t)v,v\rangle_{H^{-1},H^1_0} \geq \lambda_1 \lVert v \rVert_{H^1_0} - \lambda_2 \lVert v \rVert_{L^2}\]
for all $t$ and all $v \in H^1_0(\Omega;\R^d)$. Given $v_1,v_2 \in H^1_0(\Omega;\R^d)$, $\langle A(t)v_1,v_2\rangle_{H^{-1},H^1_0}$ has weak derivative $\langle A'(t)v_1,v_2\rangle_{H^{-1},H^1_0}$, where
\[A'(t)u = -\divo(D^3 W(\nabla y (\cdot,t))[\partial_t \nabla y(\cdot,t),\nabla u]).\]
Since $D^3 W(\nabla y)[\partial_t \nabla y] \in L^\infty([0,T_{\rm cont}+\delta] \times \Omega; \R^{d \times d \times d \times d})$, we see that $A'$ is bounded with values in $L(H^1_0(\Omega;\R^d), H^{-1}(\Omega;\R^d))$. Therefore, we can use Theorem \ref{thm:additionalregularity} to conclude that
\[\partial_t^m y \in C([0,T_{\rm cont}+\delta];L^2(\Omega;\R^d))\]
and
\[\partial_t^{m-1} y \in C([0,T_{\rm cont}+\delta];H^1(\Omega;\R^d)).\]
For $1 \leq k \leq m-2$, taking $k$ time derivatives in the equation we find
\[\partial_t^k y  = (A(t) + \lambda \Id)^{-1} (-\partial_t^{k+2} y + \lambda \partial_t^k y  + \partial_t^k f + \divo S).\]
Here
\begin{align*}
S &= \partial_t^k (DW_{\rm CB}(\nabla y)) - D^2 W_{\rm CB}(\nabla y)[\partial_t^k \nabla y]\\
  &= \sum_{\substack{
            \beta \in \N_0^{d \times d}\\
            2 \leq \vert \beta \rvert \leq k}}
            \sum_{s=1}^{k} \sum_{p_s(k,\beta)} k! D^{\beta+E_{ij}}W_{\rm CB}(\nabla y) \prod_{j=1}^s \frac{(\partial_t^{\gamma_j}\nabla y)^{\lambda_j}}{\lambda_j! \gamma_j!^{\lvert \lambda_j \rvert}}
\end{align*}
and $A(t) + \lambda \Id \colon H^{m-k} \cap H^1_0 \to H^{m-k-2}$ is invertible for $\lambda$ large enough because of Theorem \ref{thm:optimalsobolevregularity}, where we use that $D^2W_{\rm CB}(\nabla y) \in L^{\infty}([0,T_{\rm cont}+\delta],H^{m-1}(\Omega;\R^d))$ according to Lemma \ref{lem:composition}. Theorem \ref{thm:optimalsobolevregularity} also gives a time independent bound on $\lVert (A(t) + \lambda \Id)^{-1}  \rVert_{L(H^{m-k-2}, H^{m-k})}$.

According to Lemma \ref{lem:composition}, $B \mapsto D^3 W_{\rm CB} \circ B$ is a bounded map from $H^{m-2}$ to $H^{m-2}$. Therefore, we can use Lemma \ref{lem:prodofsobolev} with $M=m-2$ to see that $A'(t) \colon H^{m-k} \to H^{m-k-2}$ is well defined with
\[\lVert A'(t) \rVert_{L(H^{m-k}, H^{m-k-2})} \leq C\]
uniform in $t$. Since
\[(A(t) + \lambda \Id)^{-1} - (A(s) + \lambda \Id)^{-1} = -(A(t) + \lambda \Id)^{-1}(A(t)-A(s))(A(s) + \lambda \Id)^{-1},\]
we obtain
\[\lVert (A(t) + \lambda \Id)^{-1} - (A(s) + \lambda \Id)^{-1} \rVert_{L(H^{m-k-2},H^{m-k}\cap H^1_0)} \leq C \lvert t-s \rvert.\]
Using that $\partial_t^\gamma \nabla y$ is weakly continuous in $H^{m-1-\gamma}$, we can use Lemma \ref{lem:prodofsobolev} and its additional statement with $M=m-2$, $N=k-\lvert \beta \rvert$, $L=m-k-1 < M-N$ and $\lambda_k = \gamma_j - 1$ to find that $S$ is (strongly) continuous with values in $H^{m-k-1}$.

Putting all of this together we find inductively, starting at $k=m$ and $k=m-1$, that
\[\partial_t^k y \in C([0,T_{\rm cont}+\delta]; H^{m-k}(\Omega;\R^d))\]
for $1 \leq k \leq m$.

For $k=0$ we can no longer use the theory for linear systems in divergence form. Instead, we look at the operator
\[(A(t)u)_i = - \sum_{k,j,l} (D^2 W_{\rm CB} (\nabla y(\cdot,t)))_{ijkl} \frac{\partial^2 u_k}{\partial x_j \partial x_l} .\]
Now we have
\[y  = (A(t) + \lambda \Id)^{-1} (\partial_t^{2} y + \lambda y  - f).\]
But Theorem \ref{thm:optimalsobolevregularity} also holds in non-divergence form, hence
\[(A(t) + \lambda \Id)^{-1} \colon H^{m-2}(\Omega;\R^d) \to H^{m}(\Omega;\R^d) \cap H^1_0(\Omega;\R^d)\]
is well defined and bounded independently of $t$ since $m-2>\frac{d}{2}$ gives a bound on $D^2W_{\rm CB}(\nabla y(\cdot;t))$ in $W^{1, \infty}$. The continuity then follows along the same lines as for $k \geq 1$.

Having established the additional regularity, we have a contradiction to the definition of $T_{\rm cont}$. This proves the existence of a $T_{\rm cont}$ with the desired properties. But due to the local uniqueness of solutions, any smaller $T$ cannot satisfy either one of (i), (ii) or (iii). Therefore, $T_{\rm cont}$ is unique.
\end{proof}

\section{Atomistic Elastodynamics} \label{sec:atom}

The main theorem of this paper is the existence of a solution to the atomistic equations (for $\eps$ small enough), for as long as the corresponding solution to the Cauchy-Born continuum equations exists and is atomistically stable. But before we state and prove the theorem, let us prove two auxiliary theorems that are already interesting on their own. In both cases we will prove more general versions than what we will actually need for the main theorem.

We start with a theorem on local existence and uniqueness.
\begin{thm} \label{thm:discretelocalexistence}
Let $d \in \{1,2,3,4\}$,  $V \subset \R^{d \times \mathcal{R}}$, $W_{\rm atom} \in C^2_b(V)$ and set $V_{\rm atom} = \{A \in V \colon \lambda_{\rm atom}(A)>0\}$.
Let $\eps_0 >0$, $C_1>0$, $r_0>0$ and $\gamma \in \big[\frac{d}{2},2 \big]$, such that $4C_1 \eps_0^{\gamma-\frac{d}{2}} \leq r_0$. Furthermore, let $0<\eps \leq \eps_0$, $T_0>0$ and fix $f_{\rm atom}$, $g_{\rm atom}$, $y_{\rm ref}$ with
\begin{align*}
f_{\rm atom}(x,\cdot) &\in L^2((0,T_0); \R^d) \quad \text{for all } x \in \into_\eps \Omega,\\
g_{\rm atom}(x,\cdot) &\in H^2((0,T_0); \R^d) \quad \text{for all } x \in \partial_\eps \Omega,\\
y_{\rm ref}(x,\cdot) &\in H^2((0,T_0); \R^d) \quad \text{for all } x \in \Omega \cap \eps \Z^d,
\end{align*}
such that
\[\dist (D_{\mathcal{R},\eps} y_{\rm ref} (x,t), V_{\rm atom}^c) > r_0\]
in $\sinto_\eps \Omega \times [0,T_0]$ and
\[\sup_t \lVert y_{\rm ref}(t) -g_{\rm atom}(t) \rVert_{\partial_\eps \Omega, 0} \leq C_1 \eps^{\gamma}.\]
Then there exists a time $T>0$ which may depend on all the previous quantities, including $\eps$, such that the following holds:\\
Given any $t_0 \in [0,T_0)$ and $h_{{\rm atom},0} \in \mathcal{A}_\eps (\Omega, g_{\rm atom}(\cdot,t_0))$, $h_{{\rm atom},1} \in \mathcal{A}_\eps (\Omega, \dot{g}_{\rm atom}(\cdot,t_0))$, such that
\begin{align*}
\lVert h_{{\rm atom},1}-\dot{y}_{\rm ref}(\cdot,t_0) \rVert_{\ell^2_\eps(\into_\eps \Omega)}^2 &+ \lVert h_{{\rm atom},0} - y_{\rm ref}(\cdot,t_0) \rVert_{h^1_\eps(\sinto_\eps \Omega)}^2 \leq C_1^2 \eps^{2\gamma},
\end{align*}
there is a unique solution $y \in H^2((t_0,\min\{t_0+T,T_0\}); \R^d)^{\Omega \cap \eps \Z^d}$ to the discrete initial-boundary-value problem on $[t_0,\min\{t_0+T,T_0\}]$.
\end{thm}
\begin{proof}
This is basically the Picard-Lindelöf Theorem. But we want to quantify the dependence on the initial conditions. We look at the set
\begin{align*}
K_{T,b,z_0} = \Big\{ (z_1,z_2) &\colon z_1,z_2 \in C([t_0,\min\{ t_0+T,T_0\}];\ell^\infty(\Omega \cap \eps \Z^d))\\
& z_1(t) \in \mathcal{A}_\eps(\Omega;0), z_2(t) \in \mathcal{A}_\eps(\Omega;0)\\
& \sup_t \lVert z(t) - z^0 \rVert_{\ell^\infty(\Omega \cap \eps \Z^d)}\leq b \Big\},
\end{align*}
with the metric induced by $\lVert z \rVert =  \sup_t \lVert z(t) \rVert_{\ell^\infty(\Omega \cap \eps \Z^d)}$.
Here we substituted
\[z(t)=\begin{pmatrix} y(t) - y_{\rm ref}(t) - T_\eps (g_{\rm atom}(t) - y_{\rm ref}(t)) \\ \dot{y}(t) - \dot{y}_{\rm ref}(t) - T_\eps (\dot{g}_{\rm atom}(t) - \dot{y}_{\rm ref}(t)) \end{pmatrix}\]
 and
\[z^0=\begin{pmatrix} h_{{\rm atom},0} - y_{\rm ref}(t_0) - T_\eps (g_{\rm atom}(t_0) - y_{\rm ref}(t_0)) \\ h_{{\rm atom},1} - \dot{y}_{\rm ref}(t_0) - T_\eps (\dot{g}_{\rm atom}(t_0) - \dot{y}_{\rm ref}(t_0)) \end{pmatrix}.\]
The equation can be written as $\dot{z}(t) = F(t,z(t))$, where $F_1(t,z_1,z_2) = z_2$ and
\begin{align*}
F_2&(t,z_1,z_2)(x) = f_{\rm atom}(x,t) - \ddot{y}_{\rm ref}(x, t) - T_\eps (\ddot{g}_{\rm atom}(t) - \ddot{y}_{\rm ref}(t)) \\
&\quad +\divo_{\mathcal{R},\eps} \big( DW_{\rm atom}(D_{\mathcal{R},\eps} y_{\rm ref} (x,t) + D_{\mathcal{R},\eps} T_\eps (g_{\rm atom}(t) - y_{\rm ref}(t))(x) + D_{\mathcal{R},\eps} z_1(x))\big)
\end{align*}
for $x \in \into_\eps \Omega$, but $F_2(t,z_1,z_2)(x)=0$ for $x \in \partial_\eps \Omega$.
Since we do not even claim strong differentiability, it is best to look at the fixed point equation of
\[ G(z)(t) = z^0 + \int_{t_0}^t F(s,z(s))\,ds.\]
Clearly,
\[\sup_t \lVert y_{\rm ref}(t) -g_{\rm atom}(t) \rVert_{\partial_\eps \Omega, 0} \leq C_1 \eps^{\gamma}\]
implies
\begin{align*}
\lvert D_{\mathcal{R},\eps} T_\eps(g_{\rm atom} - y_{\rm ref}) (x,t) \rvert &\leq \eps^{- \frac{d}{2}} \lVert D_{\mathcal{R},\eps} T_\eps(g_{\rm atom} - y_{\rm ref}) (t) \rVert_{h^1_\eps}\\
&=\eps^{- \frac{d}{2}} \lVert g_{\rm atom} - y_{\rm ref} (t) \rVert_{\partial_\eps \Omega, 0}\\
&\leq C_1 \eps^{\gamma - \frac{d}{2}}
\end{align*}
uniformly in $x$ and $t$. Now, if $0<b\leq \frac{\eps r_0}{8 \lvert \mathcal{R} \rvert^{\frac{1}{2}}}$
then for any $z \in K_{T,b,z_0}$
\begin{align*}
&\lvert D_{\mathcal{R},\eps} T_\eps(g_{\rm atom} - y_{\rm ref}) (x,t)+ D_{\mathcal{R},\eps}z_1(x,t) \rvert\\
&\quad \leq \lvert D_{\mathcal{R},\eps} T_\eps(g_{\rm atom} - y_{\rm ref}) (x,t) \rvert + \lvert D_{\mathcal{R},\eps} (z_1(x,t)-z_1^0(x)) \rvert + \lvert D_{\mathcal{R},\eps}z_1^0(x) \rvert \\
&\quad \leq C_1 \eps^{\gamma - \frac{d}{2}} + \frac{2 b \lvert \mathcal{R}\rvert^{\frac{1}{2}}}{\eps}+ \lvert D_{\mathcal{R},\eps}z_1^0(x) \rvert \\
&\quad \leq 3 C_1 \eps^{\gamma - \frac{d}{2}} +  \frac{2 b \lvert \mathcal{R}\rvert^{\frac{1}{2}}}{\eps}\\
&\quad \leq r_0.
\end{align*}
Therefore $F(s,z(s))$ is well defined. Furthermore,
\begin{align*}
\sup\limits_t \sup\limits_{x \in \Omega \cap \eps \Z^d} &\lvert G(z)(x,t)-z^0(x) \rvert\\ &\leq \sup\limits_{t,x} \int_{t_0}^t \lvert F_1(s,z(s)) \rvert + \lvert F_2(s,z(s)) \rvert \,ds\\
&\leq bT + T C_1 \eps^{\gamma - \frac{d}{2}} + T \lVert T_\eps ( \dot{g}_{\rm atom} - \dot{y}_{\rm ref})\rVert_{L^\infty(0,T_0;\ell^\infty)} + T\frac{2 \lvert \mathcal{R}\rvert}{\eps} \lVert DW_{\rm atom}\rVert_\infty\\
&\quad + \sqrt{T} \big( \lVert f \rVert_{L^2(0, T_0; \ell^\infty)} + \lVert \ddot{y}_{\rm ref} \rVert_{L^2(0, T_0; \ell^\infty)} + \big\lVert T_\eps (\ddot{g}_{\rm atom} - \ddot{y}_{\rm ref}) \big\lVert_{L^2(0,T_0; \ell^\infty)}\big).
\end{align*}
In particular, for $T$ small enough
\[\sup\limits_t \sup\limits_{x \in \Omega \cap \eps \Z^d} \lvert G(z)(x,t)-z^0(x) \rvert \leq b.\]
Since $G(z)$ also has the correct boundary values, $G \colon K_{T,b,z_0} \to K_{T,b,z_0}$ is well defined. Given $z,\tilde{z} \in K_{T,b,z_0}$ we calculate
\begin{align*}
\sup\limits_t &\sup\limits_{x \in \Omega \cap \eps \Z^d} \lvert G(z)(x,t) - G(\tilde{z})(x,t) \rvert
\leq T \sup\limits_t \lVert F(t,z(t)) - F(t,\tilde{z}(t)) \rVert_{\ell^\infty(\Omega \cap \eps \Z^d)}\\
&\leq T\big( \sup_t \lVert z_2(t) - \tilde{z}_2(t) \rVert_{\ell^\infty(\Omega \cap \eps \Z^d)} + \lvert \mathcal{R}\rvert^{\frac{3}{2}} \frac{4}{\eps^2} \lVert D^2W_{\rm atom} \rVert_\infty \sup_t \lVert z_1(t) - \tilde{z}_1(t) \rVert_{\ell^\infty(\Omega \cap \eps \Z^d)}\big)\\
&\leq T\big(1+\lvert \mathcal{R}\rvert^{\frac{3}{2}} \frac{4}{\eps^2} \lVert D^2W_{\rm atom} \rVert_\infty \big) \sup_t \lVert z(t) - \tilde{z}(t) \rVert_{\ell^\infty(\Omega \cap \eps \Z^d)}\\
&\leq \frac{1}{2}\sup_t \lVert z(t) - \tilde{z}(t) \rVert_{\ell^\infty(\Omega \cap \eps \Z^d)},
\end{align*}
if we also require
\[T \leq \frac{1}{2+2\lvert \mathcal{R}\rvert^{\frac{3}{2}} \frac{4}{\eps^2} \lVert D^2W_{\rm atom} \rVert_\infty}.\]
Now we can use the Banach fixed point theorem. If $b$ and $T$ satisfy the constraints above, then $G$ has a unique fixed point $z \in K_{T,b,z_0}$. Setting $y = z_1 + y_{\rm ref} + T_\eps (g_{\rm atom}(t) - y_{\rm ref}(t))$, we have
\[y \in H^2((t_0,\min\{ t_0+T,T_0\});\ell^\infty(\Omega \cap \eps \Z^d))\]
and $y$ solves the discrete initial-boundary-value problem in the absolutely continuous sense on $[t_0,\min\{t_0+T,T_0\}]$.
Now conversely, if $y$ is any solution in $H^2((t_0,\min\{ t_0+T,T_0\});\ell^\infty(\Omega \cap \eps \Z^d))$ that satisfies
\[D_{\mathcal{R},\eps} y(x,t) \in V\]
for all $t$ and $x \in \sinto_\eps \Omega$, we can substitute back to $z$ and calculate
\begin{align*}
&\lVert z(t) - z^0 \rVert_{\ell^\infty(\Omega \cap \eps \Z^d)} \leq \int_{t_0}^t \lVert z(s) - z^0 \rVert_{\ell^\infty(\Omega \cap \eps \Z^d)}\,ds + \frac{2 \mathcal{R} \lVert DW_{\rm atom} \rVert_\infty}{\eps} (t-t_0)\\
&\qquad+ \sqrt{t-t_0} \big( \lVert f \rVert_{L^2(0, T_0; \ell^\infty)} + \lVert \ddot{y}_{\rm ref} \rVert_{L^2(0, T_0; \ell^\infty)} + \big\lVert T_\eps (\ddot{g}_{\rm atom} - \ddot{y}_{\rm ref}) \big\lVert_{L^2(0,T_0; \ell^\infty)} \big)\\
&\leq \int_{t_0}^t \lVert z(s) - z^0 \rVert_{\ell^\infty(\Omega \cap \eps \Z^d)}\,ds +C_1\big( (t-t_0)+ \sqrt{t-t_0} \big)
\end{align*}

Using Grönwall's inequality we thus get
\begin{align*}
\lVert z(t) - z^0 \rVert_{\ell^\infty(\Omega \cap \eps \Z^d)} &\leq C_1(t-t_0 + \sqrt{t-t_0})e^{t-t_0}\\
&\leq 2C_1\sqrt{T}e^T\\
&\leq b
\end{align*}
if we additionally assume $T \leq 1$ and 
$T \leq \frac{b}{2 C_1 e}$.
Therefore, $z \in K_{T,b,z_0}$, and the uniqueness of the solution follows.
\end{proof}

Although this lemma already gives us a local solution, the time $T$ depends heavily on $\eps$ and is not necessarily bounded from below as $\eps$ goes to $0$. One of our main goals is to show existence on an $\eps$-independent time interval. Actually, we even want to go one step further. We will show that the atomistic solution exists as long as the solution to the continuous problem exists and is atomistically stable.

As mentioned in the introduction establishing an atomistic Gårding inequality is key to provide control of the stability of solutions for long times and large deformations. There are some differences to the continuous Gårding inequality (Theorem \ref{thm:contGårding}). Unsurprisingly, we need to require atomistic stability. Due to the discreteness of the problem we also need to track the variation of the coefficients and the dependence on $\eps$ more explicitly.

\begin{thm} \label{thm:discreteGårding}
Let $d \in \N$, $\Omega \subset \R^d$ open and bounded, and $\lambda_1, \Lambda, \eps_0>0$. Consider a family $A_\eps \colon \sinto_\eps \Omega \to \R^{d \times \mathcal{R} \times d \times \mathcal{R}}$, for $0< \eps \leq \eps_0$, with $\lambda_{\rm atom} (A_\eps(x)) \geq \lambda_1$ for all $x \in \sinto_\eps \Omega$ and $0< \eps \leq \eps_0$. Assume also that $ \sup_\eps \lVert A_\eps \rVert_\infty \leq \Lambda$ and that there are $r_\eps \geq \eps$ such that
\[\sup_{0<\eps \leq \eps_0} \sup_{\substack{x,x' \in \sinto_\eps \Omega \\ \lvert x-x' \rvert \leq 2 r_\eps+ 2 \eps R_{\rm max}}} \lvert A_\eps(x) - A_\eps(x') \rvert \leq \frac{\lambda_1}{4},\]
then there is a $\lambda_2=\lambda_2(\lambda_1, \Lambda, d, \mathcal{R})$, such that
\[ \eps^d \sum\limits_{x \in \sinto_\eps \Omega} A_\eps(x)[D_{\mathcal{R},\eps} u (x),D_{\mathcal{R},\eps} u (x)] \geq \frac{\lambda_1}{2} \lVert u \rVert^2_{h^1_\eps(\sinto_\eps \Omega)} - \frac{\lambda_2}{r_\eps^2} \lVert u \rVert^2_{\ell_\eps^2(\into_\eps \Omega)}\]
for all $u \in \mathcal{A}_\eps(\Omega,0)$ and $0< \eps \leq \eps_0$.
\end{thm}

\begin{rem}
In this paper we will only use the theorem in the case where $r_\eps$ is independent of $\eps$. This corresponds to $A_\eps$ only changing on the macroscopic scale. We will still prove the more general version since the theorem has some interest itself.
\end{rem}

\begin{proof}
By the definition of atomistic stability we have
\[ \eps^d \sum\limits_{x \in \sinto_\eps \Omega} A_\eps(z) [D_{\mathcal{R},\eps} u (x),D_{\mathcal{R},\eps} u (x)] \geq \lambda_1 \lVert u \rVert^2_{h^1_\eps(\sinto_\eps \Omega)}\]
for every $z \in \sinto_\eps \Omega$, every $\eps>0$, and every $u \in \mathcal{A}_\eps(\Omega,0)$.

Now, choose countable many $z_j \in \R^d$ and $\eta_j \in C_c^\infty(\R^d;[0,1])$ such that $\sum_j \eta_j^2(x) = 1$ for every $x \in \R^d$, $\supp \eta_j \subset B_{r_\eps}(z_j)$, $\lvert \nabla \eta_j \rvert \leq \frac{C(d)}{r_\eps}$, and the decomposition is locally finite in the sense that 
\[  \lvert \{j \colon B_{r_\eps}(z_j) \cap B_R(x) \neq \emptyset \} \rvert \leq C(d) \big(1+\frac{R}{r_\eps}\big)^d\]
for all $x \in \R^d$ and $R>0$. Whenever $B_{r_\eps + \eps R_{\rm max}}(z_j)\cap \sinto_\eps \Omega \neq \emptyset$ fix a point $x_{j,\eps} \in B_{r_\eps + \eps R_{\rm max}}(z_j)\cap \sinto_\eps \Omega$. By assumption we then have $\lvert A_\eps (x_{j, \eps}) - A_\eps(x) \rvert \leq \frac{\lambda_1}{4}$ for every $x \in B_{r_\eps + \eps R_{\rm max}}(z_j)\cap \sinto_\eps \Omega$.
Now, since
\[(D_{\mathcal{R},\eps} (\eta_j u) (x))_\rho = \eta_j(x)(D_{\mathcal{R},\eps} u(x))_\rho + u(x+\eps \rho) (D_{\mathcal{R},\eps} \eta_j (x))_\rho \]
for any $\delta>0$ we can calculate with Young's inequality
\begin{align*}
\eps^d \sum\limits_{x \in \sinto_\eps \Omega} &A_\eps(x) [D_{\mathcal{R},\eps} u (x),D_{\mathcal{R},\eps} u (x)] \\
&= \eps^d \sum_j \sum\limits_{x \in \sinto_\eps \Omega} A_\eps(x) [\eta_j(x) D_{\mathcal{R},\eps} u (x),\eta_j(x) D_{\mathcal{R},\eps} u (x)] \\
&\geq \eps^d \sum_j \sum\limits_{x \in \sinto_\eps \Omega} A_\eps(x) [D_{\mathcal{R},\eps} (\eta_j u) (x),D_{\mathcal{R},\eps} (\eta_j u)  (x)]\\
&\qquad - \delta \eps^d \sum_j \sum\limits_{x \in \sinto_\eps \Omega} n_j^2(x) \lvert D_{\mathcal{R},\eps} u (x) \rvert^2\\
&\qquad - \Lambda (1+ \frac{\Lambda}{\delta})\eps^d \sum_j \sum\limits_{x \in \sinto_\eps \Omega} \sum_\rho \lvert u(x+\eps \rho) \rvert^2 \Big\lvert \frac{\eta_j(x+\eps \rho) - \eta_j(x)}{\eps}\Big\rvert^2\\
&\geq \eps^d \sum_j \sum\limits_{x \in \sinto_\eps \Omega} A_\eps(x_{j,\eps}) [D_{\mathcal{R},\eps} (\eta_j u) (x),D_{\mathcal{R},\eps} (\eta_j u)  (x)] - \frac{\lambda_1}{4} \lvert D_{\mathcal{R},\eps} (\eta_j u) (x) \rvert^2 \\
&\qquad - \delta \eps^d \sum\limits_{x \in \sinto_\eps \Omega} \lvert D_{\mathcal{R},\eps} u (x) \rvert^2\\
&\qquad - \Lambda (1+ \frac{\Lambda}{\delta})\eps^d \sum_j \sum\limits_{x \in \sinto_\eps \Omega} \sum_\rho \lvert u(x+\eps \rho) \rvert^2 \Big\lvert \frac{\eta_j(x+\eps \rho) - \eta_j(x)}{\eps}\Big\rvert^2.
\end{align*}
Using the atomistic stability at $x_{j,\eps}$, we can continue in the spirit to find
\begin{align*}
\eps^d \sum\limits_{x \in \sinto_\eps \Omega} &A_\eps(x) [D_{\mathcal{R},\eps} u (x),D_{\mathcal{R},\eps} u (x)] \\
&\geq \eps^d \frac{3}{4} \lambda_1 \sum_j \sum\limits_{x \in \sinto_\eps \Omega}  \lvert D_{\mathcal{R},\eps} (\eta_j u) (x) \rvert^2\\
&\qquad - \delta \eps^d \sum\limits_{x \in \sinto_\eps \Omega} \lvert D_{\mathcal{R},\eps} u (x) \rvert^2\\
&\qquad - \Lambda (1+ \frac{\Lambda}{\delta})\eps^d \sum_j \sum\limits_{x \in \sinto_\eps \Omega} \sum_\rho \lvert u(x+\eps \rho) \rvert^2 \Big\lvert \frac{\eta_j(x+\eps \rho) - \eta_j(x)}{\eps}\Big\rvert^2\\
&\geq \eps^d (\frac{3}{4} \lambda_1-2 \delta) \sum\limits_{x \in \sinto_\eps \Omega}  \lvert D_{\mathcal{R},\eps} u (x) \rvert^2 \\
&\qquad - \Big(\Lambda(1 + \frac{\Lambda}{\delta}) + \frac{3}{4} \lambda_1 (1+ \frac{3 \lambda_1}{4\delta}) \Big)  \lVert u \rVert_{\ell^2_\eps(\into_\eps \Omega)}^2 C(d, \mathcal{R}) \frac{1}{r_\eps^2} \Big(1+ \frac{\eps R_{\rm max}}{r_\eps}\Big)^d
\end{align*}
Now, choosing $\delta = \frac{\lambda_1}{8}$ and using $r_\eps \geq \eps$, we indeed get
\begin{align*}
\eps^d \sum\limits_{x \in \sinto_\eps \Omega} &A_\eps(x) [D_{\mathcal{R},\eps} u (x),D_{\mathcal{R},\eps} u (x)] \\
&\geq \frac{\lambda_1}{2}\eps^d \sum\limits_{x \in \sinto_\eps \Omega}  \lvert D_{\mathcal{R},\eps} u (x) \rvert^2 - C(\lambda_1, \Lambda, d, \mathcal{R}) \frac{1}{r_\eps^2} \lVert u \rVert_{\ell^2_\eps(\into_\eps \Omega)}^2.
\end{align*}
\end{proof}

Let us make some last preparations for our main theorem. We will show that there are atomistic solutions close to the extended and regularized reference configuration
\[y_{\rm ref} = \eta_\eps \ast (E y_{\rm cont}) \]
where $y_{\rm cont}$ is a solution of the continuous problem, $\eta_\eps(x)$ denotes the standard scaled mollifying kernel, and $E$ denotes the Stein extension which is an extension operator for all Sobolev spaces requiring only very little regularity of the boundary, cf. \cite[Chapter VI]{steinsingint}.

The conditions that we will pose on the time-dependent atomistic boundary conditions can be formulated much easier with the following norm. Given $g \colon \partial_\eps \Omega \times [0,T_0]$, such that $g(x, \cdot) \in H^2(0,T_0)$ for all $x \in \partial_\eps \Omega$, we look at the (quadratic) functional
\begin{align*}
\mathcal{F}(z) &= \lVert z(0) \rVert_{\ell^2_\eps(\into_\eps \Omega)}^2 + \lVert z(0) \rVert_{h^1_\eps(\sinto_\eps\Omega)}^2 + \lVert \dot{z}(0) \rVert_{\ell^2_\eps(\into_\eps \Omega)}^2\\
&\quad + \int_0^{T_0} \lVert z(\tau) \rVert_{h^1_\eps(\sinto_\eps\Omega)}^2 + \lVert \dot{z}(\tau) \rVert_{h^1_\eps(\sinto_\eps\Omega)}^2 + \lVert \ddot{z}(\tau) \rVert_{\ell^2_\eps(\into_\eps \Omega)}^2 \,d\tau
\end{align*}
for $z \colon \Omega \cap \eps \Z^d \times [0,T_0]$, such that $z(x, \cdot) \in H^2(0,T_0)$ for all $x \in \Omega \cap \eps \Z^d$ and $z |_{\partial_\eps \Omega \times [0, T_0]} = g$. Clearly the functional is lower semi-continuous and coercive in $H^2$ and thus has a minimizer. By strict convexity this minimizer is unique and it is also given as the unique solution to 
\begin{align*}
0&=(z(0), w(0))_{\ell^2_\eps}^2 + (z(0), w(0))_{h^1_\eps}^2 + (\dot{z}(0), \dot{w}(0))_{\ell^2_\eps}^2\\
&\quad + \int_0^{T_0} (z(\tau), w(\tau))_{h^1_\eps}^2 + (\dot{z}(\tau), \dot{w}(\tau))_{h^1_\eps}^2 + (\ddot{z}(\tau), \ddot{w}(\tau))_{\ell^2_\eps}^2 \,d\tau
\end{align*}
for all $w \in H^2$ with $w |_{\partial_\eps \Omega} = 0$. In particular, the mapping $K_\eps$ that maps $g$ to this minimizer is linear. Furthermore, $\lVert g \rVert_{\partial_\eps \Omega,dyn} := \big(\mathcal{F}(K_\eps g) \big)^{\frac{1}{2}}$ is a norm. Besides dominating the norms used in its definition, we will also use that
\[\lVert K_\eps g \rVert_{L^\infty(0,T_0; h^1_\eps)} \leq \lVert g \rVert_{\partial_\eps \Omega,dyn}\]
and
\[\lVert K_\eps g \rVert_{W^{1,\infty}(0,T_0; \ell^2_\eps)} \leq C(T) \lVert g \rVert_{\partial_\eps \Omega,dyn}.\]
We will then require
\[\lVert y_{\rm ref}-g_{\rm atom} \rVert_{\partial_\eps \Omega,dyn} \leq C_g \eps^\gamma, \]
in our main theorem below for some convergence rate $\gamma \in (\frac{d}{2},2]$.

While this specific norm is mainly chosen to satisfy certain inequalities in the proof, it is not at all surprising. The terms at the starting time are obviously required by the convergence estimate we want to prove uniformly in time (see below). The terms controlling the $h^1_\eps$-norm are crucial. Among other things, they ensure the uniform convergence of the gradients. Therefore, at the boundary, the atomistic boundary conditions enforce not only the correct asymptotic boundary values but also the correct asymptotic (normal) derivative and thus suppress surface relaxation effects. This is important for the Cauchy-Born rule to hold near the boundary. At last, a difference in the second time derivatives has a similar effect as a difference in the body forces and thus, unsurprisingly, we want both terms to be small in the same norm.

\begin{thm} \label{thm:atomisticwavethm}
Let $d \in \{2,3\}$ and $m \in \N$, $m \geq 4$. Let $T_0>0$ and let $\Omega \subset \R^d$ be an open, bounded set with $\partial \Omega$ of class $C^m$. Let $V \subset \R^{d \times \mathcal{R}}$ be open and $W_{\rm atom} \in C_b^{m+1}(V)$. Let $f$ be a continuous body force, $h_0,h_1$ initial data and $g$ boundary values such that
\begin{align*}
&f \in C^{m-1}(\overline{\Omega}\times [0,T_0]; \R^d)\\
&g \in C^{m+1}(\overline{\Omega}\times [0,T_0]; \R^d)\\
&h_0 \in H^m (\Omega; \R^d)\\
&\{(\nabla h_0(x)\rho)_{\rho \in \mathcal{R}} \colon  x \in \overline{\Omega}\} \subset V \cap \{A \colon \lambda_{\rm atom}(A)>0\} \\
&h_1 \in H^{m-1} (\Omega; \R^d)
\end{align*}
and such that the compatibility conditions of order $m$ are satisfied. Furthermore, assume that the unique solution of the Cauchy-Born problem $y_{\rm cont}$ from Theorem \ref{thm:localcontwave} exists until $T_0$ and satisfies
\[ y_{\rm cont} \in \bigcap_{k=0}^m C^k\big([0,T_0]; H^{m-k}(\Omega;\R^d)\big),\]
as well as
\[\{(\nabla y_{\rm cont}(x,t) \rho)_{\rho \in \mathcal{R}} \colon  x \in \overline{\Omega}, t \in [0,T_0]\} \subset V \cap \{A \colon \lambda_{\rm atom}(A)>0\}.\]
Now let $C_g, C_f, C_h>0$ and $\gamma \in (\frac{d}{2}+ \frac{1}{m-1},2]$. Then there is an $\eps_0>0$ such that the following holds for every $0< \eps \leq \eps_0$.

Given atomistic data $f_{\rm atom} \colon \into_\eps \Omega \times [0,T_0] \to \R^d$, $g_{\rm atom} \in H^2((0,T_0); \R^d)^{\partial_\eps \Omega}$, $h_{{\rm atom},0} \in \mathcal{A}_\eps (\Omega, g_{\rm atom}(\cdot,0))$, and $h_{{\rm atom},1} \in \mathcal{A}_\eps (\Omega, \dot{g}_{\rm atom}(\cdot,0))$ with

\[\lVert y_{\rm ref}-g_{\rm atom} \rVert_{\partial_\eps \Omega,dyn} \leq C_g \eps^\gamma, \]
%
\[\lVert h_{\rm atom, 1}- \dot{y}_{\rm ref}(0) \rVert^2_{\ell^2_\eps(\into_\eps \Omega)} + \lVert h_{\rm atom, 0}- y_{\rm ref}(0) \rVert^2_{\ell^2_\eps(\into_\eps \Omega)} + \lVert h_{\rm atom, 0}- y_{\rm ref}(0) \rVert^2_{h^1_\eps(\into_\eps \Omega)} \leq C_h^2 \eps^{2\gamma},\]
\[\lVert f_{\rm ref}-f_{\rm atom}  \rVert_{L^2(0,T_0; \ell^2_\eps(\into \Omega))} \leq C_f \eps^\gamma,\]
where
\[f_{\rm ref} = \tilde{f} + \ddot{y}_{\rm ref} - \ddot{\tilde{y}}_{\rm cont}.\]

Then there is a unique $y \in H^2((0,T_0); \R^d)^{\Omega \cap \eps \Z^d}$ that solves the atomistic equations with body force $f_{\rm atom}$ boundary values $g_{\rm atom} $ and initial conditions $h_{{\rm atom},0}, h_{{\rm atom},1}$. Furthermore, we have the convergence estimate
\begin{align*}
&\lVert \dot{y} - \dot{y}_{\rm ref} \rVert_{\ell^2_\eps(\into_\eps \Omega)} +\lVert y - y_{\rm ref} \rVert_{h^1_\eps(\sinto_\eps \Omega)}+ \lVert y- y_{\rm ref} \rVert_{\ell^2_\eps(\into_\eps \Omega)}\\
&\qquad \leq C e^{Ct} (C_g + C_h + C_f + \eps^{2-\gamma}) \eps^{\gamma}
\end{align*}
for some $C=C(\mathcal{R}, V, W_{\rm atom}, y_{\rm cont}, \Omega, m, \gamma)>0$.
\end{thm}
\begin{rem}
Remember that
\[\tilde{f}(x)= \fint_{Q_\eps(x)} f(z)\,dz.\]
If $y_{\rm cont} \in H^2(0,T;C^{1,\gamma-1}(\Omega; \R^d))$, the more natural choice $f_{\rm ref} = \tilde{f}$ suffices since then
\[\lVert \ddot{y}_{\rm ref} - \ddot{\tilde{y}}_{\rm cont} \rVert_{L^2(0,T_0; \ell^2_\eps(\into \Omega))} \leq C(y_{\rm cont}) \eps^\gamma.\]
This condition is automatically satisfied if $m \geq 6$.
\end{rem}
\begin{proof}
First let us prove that $Ey_{\rm cont}$ and $y_{\rm ref}$ inherit the atomistic stability from $y_{\rm cont}$ as long as we stay in or close to $\Omega$.
Given $R>0$ and $x \in \Omega + B_{\eps R}(0)$, take $x' \in \Omega$ with $\lvert x-x'\rvert \leq R \eps$. Then we directly see
\begin{align*}
\lvert \nabla Ey_{\rm cont}(x) - \nabla y_{\rm cont}(x') \rvert &\leq \lVert \nabla^2 Ey_{\rm cont} \rVert_{L^\infty} R \eps \\
&\leq C(\Omega) R \eps \lVert \nabla^2 y_{\rm cont} \rVert_{L^\infty}
\end{align*}
since $y_{\rm cont} \in H^4(\Omega)$, which embeds into $W^{2, \infty}$ and even $C^2$ for $d \leq 3$. It immediately follows that
\[\lvert \nabla y_{\rm ref}(x) - \nabla y _{\rm cont}(x')\rvert \leq C(\Omega) (R+1) \eps \lVert \nabla^2 y_{\rm cont} \rVert_{L^\infty}\]
and
\begin{align*}
\lvert D_{\mathcal{R},\eps} y_{\rm ref}(x) - (\nabla y _{\rm cont}(x')\rho)_{\rho \in \mathcal{R}}\rvert &= \Big(\sum_\rho \Big\lvert \int_0^1 \nabla y_{\rm ref}(x+s\eps \rho)\rho - \nabla y_{\rm cont}(x')\rho \,ds \Big\rvert^2\Big)^{\frac{1}{2}}\\
&\leq C(\Omega, \mathcal{R}, R) \eps \lVert \nabla^2 y_{\rm cont} \rVert_{L^\infty}\\
&= C(\Omega, \mathcal{R}, R, y_{\rm cont}) \eps 
\end{align*}
Since the stability constant is continuous, the set $\{A \in V \colon \lambda_{\rm atom}(A)>0\}$ is open. On the other hand, $\{(\nabla y_{\rm cont}(x,t) \rho)_{\rho \in \mathcal{R}} \colon  x \in \overline{\Omega}, t \in [0,T]\}$ is compact. Therefore,
\[\{(\nabla y_{\rm cont}(x,t) \rho)_{\rho \in \mathcal{R}} \colon  x \in \overline{\Omega}, t \in [0,T]\} + \overline{B_{\eps C}(0)} \subset \{A \in V \colon \lambda_{\rm atom}(A)>0\} \]
for all $\eps \leq \eps_0$ if $\eps_0 = \eps_0(\mathcal{R}, \Omega, y_{\rm cont}, R, V, W_{\rm atom})$ is chosen small enough.

For a time dependent atomistic deformation we define the norm-energy
\begin{align*}
\mathcal{E}(t) =\lVert \dot{u} \rVert^2_{\ell^2_\eps(\into_\eps \Omega)} +\lVert u \rVert^2_{h^1_\eps(\sinto_\eps \Omega)}+ \lVert u \rVert^2_{\ell^2_\eps(\into_\eps \Omega)},
\end{align*}
where $u = y- y_{\rm ref} - K_\eps (g_{\rm atom}- y_{\rm ref})$. Note that this energy is well-defined and continuous on $[a,b]$ if $u \in H^2((a,b); \R^d)^{\Omega \cap \eps \Z^d}$.

For $B>0$ to be defined later, let $T_\eps$ be the supremum of all times $T \leq T_0$ such that a solution $y$ exists on $[0,T)$ and
\[\mathcal{E}(t) \leq B^2 \eps^{2\gamma}\]
for $t \in [0,T]$.

Note that
\[\sup_t \lVert y_{\rm ref}(t) - g_{\rm atom}(t)\rVert_{\partial_\eps\Omega,0} \leq \lVert y_{\rm ref} - g_{\rm atom}\rVert_{\partial_\eps \Omega,dyn} \leq C_g \eps^\gamma.\]
Choosing $\eps_0$ so small that
\[4(\max\{C_g,C_h\} +1) \eps_0^{\gamma - \frac{d}{2}} \leq \inf_{x,t} \dist(D_{\mathcal{R},\eps} y_{\rm ref}(x,t), V_{\rm atom}^c)\]
we can apply the local result, Theorem \ref{thm:discretelocalexistence}. If furthermore $B> \sqrt{2C_g^2 + 2C_h^2}$, which will be the case in our choice of $B$, then we indeed see that $T_\eps > 0$. The uniqueness part of Theorem \ref{thm:discretelocalexistence} implies that all such solutions agree on the intersection of their domains of definition. Putting these solutions together we thus have a $y$ on $(0,T_\eps)$ such that for every $0<T<T_\eps$ it holds that $y \in H^2(0,T)$ and $y$ is a solution of the problem. If we choose $\eps_0$ even smaller, such that
\[4(\sqrt{2B^2+2C_g^2} +1) \eps_0^{\gamma - \frac{d}{2}} \leq \inf_{x,t} \dist(D_{\mathcal{R},\eps} y_{\rm ref}(x,t), V_{\rm atom}^c)\]
we can again apply Theorem \ref{thm:discretelocalexistence} with $t_0 \in (0,T_\eps)$ and initial conditions $y(t_0), \dot{y}(t_0)$, since
\begin{align*}
\lVert \dot{y}(t_0) - \dot{y}_{\rm ref}(t_0)  \rVert^2_{\ell^2_\eps(\into_\eps \Omega)} +\lVert  y(t_0) - y_{\rm ref}(t_0)  \rVert^2_{h^1_\eps(\sinto_\eps \Omega)} &\leq 2 \mathcal{E}(t) + 2 C_g^2 \eps^{2 \gamma}\\
&\leq (2 B^2 + 2 C_g^2) \eps^{2 \gamma}.
\end{align*}
We thus get a solution on $(t_0, \max\{t_0 + T_{\rm loc}, T_0\})$ for some $T_{\rm loc}$ independent of $t_0$. Again by uniqueness all solutions fit together. Therefore, $y \in H^2(0,T_\eps)$ and $y$ is a solution of the problem on $(0,T_\eps)$. Additionally, we know that $T_\eps = T_0$ or the solution exists on a larger intervall than $(0,T_\eps)$. In the second case we must have $\mathcal{E}(T_\eps) = B^2 \eps^{2\gamma}$. To ensure that we are in the first case it thus suffices to estimate the energy on $[0,T_\eps]$. This is what we will do in the rest of the proof.

The energy bound implies
\[\lVert D_{\mathcal{R},\eps} y- D_{\mathcal{R},\eps} y_{\rm ref} \rVert_\infty \leq (C_g +B)\eps^{\gamma - \frac{d}{2}}.\]
Choosing $\eps_0$ even smaller, now also depending on $C_g$, $B$ and $\gamma$, by continuity of the stability constant, we can find a $\lambda_0=\lambda_0(y_{\rm cont}, V, W_{\rm atom})>0$ such that
\[\lambda_{\rm atom}(M)\geq \lambda_0 \quad \text{and} \quad M \in V\]
for all $M$ with $\lvert M - D_{\mathcal{R},\eps} y_{\rm ref} \rvert \leq (C_g + B) \eps_0^{\gamma - \frac{d}{2}}$ for any $x,t$. In particular, we see that this is true for $M = D_{\mathcal{R},\eps} y$ or $M=s D_{\mathcal{R},\eps} y + (1-s)D_{\mathcal{R},\eps} y_{\rm ref}$, $s\in [0,1]$ as long as $t < T_\eps$.

Setting 
\[A_\eps = \int_0^1 D^2 W_{\rm atom}\big(D_{\mathcal{R},\eps}y_{\rm ref} + s(D_{\mathcal{R},\eps}y -D_{\mathcal{R},\eps}y_{\rm ref} )\big) \,ds,\]
we see that for $\lvert x-x'\rvert \leq 2r + 2 \eps R_{\rm max}$
\begin{align*}
\lvert A_\eps(x) - A_\eps(x') \rvert &\leq \lVert D^3 W_{\rm atom} \rVert_\infty \big( \lVert D^2y_{\rm ref} \rVert_\infty \lvert x-x' \rvert + 2 (B+C_g) \eps^{\gamma - \frac{d}{2}}\big) \\
&\leq C ( r + \eps + (B + C_g) \eps^{\gamma - \frac{d}{2}}).
\end{align*}
If again $\eps_0$ is small enough we can therefore use the atomistic Gårding inequality from Theorem  \ref{thm:discreteGårding} with $r=r(y_{\rm cont},W_{\rm atom}, \lambda_0)$ small enough and independent of $\eps$ to get
\begin{align*}
\mathcal{E}(t)&\leq \lVert u \rVert^2_{\ell^2_\eps(\into_\eps \Omega)} +  \lVert \dot{u} \rVert^2_{\ell^2_\eps(\into_\eps \Omega)} + \max\{2 , \frac{\lambda_0}{2}\} \lVert u \rVert^2_{h^1_\eps(\sinto_\eps \Omega)}\\
&\leq C \lVert u \rVert^2_{\ell^2_\eps(\into_\eps \Omega)} +  \max\{\frac{4}{\lambda_0} , 1\} \lVert \dot{u} \rVert^2_{\ell^2_\eps(\into_\eps \Omega)}  \\
&\quad + \max\{\frac{4}{\lambda_0} , 1\} \eps^d \sum_{x \in \sinto_\eps \Omega} A_\eps(x,t)[D_{\mathcal{R},\eps}u]^2\\
&\leq C \lVert u \rVert^2_{\ell^2_\eps(\into_\eps \Omega)} + \max\{\frac{8}{\lambda_0} , 2\} \Big( \frac{1}{2}\lVert \dot{u} \rVert^2_{\ell^2_\eps(\into_\eps \Omega)} +\frac{1}{2} \eps^d \sum_{x \in \sinto_\eps \Omega} A_\eps(x,t)[D_{\mathcal{R},\eps}u ]^2 \Big)
\end{align*}
for some $C=C(y_{\rm cont}, W_{\rm atom}, \lambda_0, \mathcal{R})$. If we rewrite this in terms of the initial conditions and take absolute values, we get 
\begin{align*}
\mathcal{E}(t) &\leq C \Big( \lVert h_{\rm atom, 1}- \dot{y}_{\rm ref}(0) - \frac{\partial}{\partial t}K_\eps(y- y_{\rm ref})(0) \rVert^2_{\ell^2_\eps(\into_\eps \Omega)}\\
&\quad + \lVert h_{\rm atom, 0}- y_{\rm ref}(0) - K_\eps(y- y_{\rm ref})(0) \rVert^2_{\ell^2_\eps(\into_\eps \Omega)}\\
&\quad + \lVert h_{\rm atom, 0}- y_{\rm ref}(0) - K_\eps(y- y_{\rm ref})(0)  \rVert^2_{h^1_\eps(\sinto_\eps \Omega)}\\
&\quad+ \Big\lvert \int_0^t (u, \dot{u})_{\ell^2_\eps} \,d\tau \Big\rvert \\
&\quad+ \Big\lvert \int_0^t (\dot{u}, \ddot{u})_{\ell^2_\eps} + \eps^d \sum_{x \in \sinto_\eps \Omega} A_\eps(x,\tau)[D_{\mathcal{R},\eps}u, D_{\mathcal{R},\eps}\dot{u} ]\\
&\quad+ \frac{1}{2}\eps^d \sum_{x \in \sinto_\eps \Omega} \dot{A}_\eps(x,\tau)[D_{\mathcal{R},\eps}u]^2 \,d\tau \Big\rvert\Big).
\end{align*}

Using our assumptions at $t=0$ and for the boundary conditions we can continue by
\begin{align*}
\mathcal{E}(t) &\leq C \Big( (C_g^2+ C_h^2) \eps^{2 \gamma} + \int_0^t \mathcal{E}(\tau) \,d\tau + \Big\lvert \int_0^t  \eps^d \sum_{x \in \sinto_\eps \Omega} \dot{A}_\eps(x,\tau)[D_{\mathcal{R},\eps}u]^2 \,d\tau \Big\rvert \\
&\quad + \Big\lvert \int_0^t \big(\dot{u}, \frac{\partial^2}{\partial t^2}K_\eps (g_{\rm atom} - y_{\rm ref})\big)_{\ell^2_\eps} \,d\tau \Big\rvert\\
&\quad + \Big\lvert \int_0^t \eps^d \sum_{x \in \sinto_\eps \Omega} A_\eps(x,\tau)[D_{\mathcal{R},\eps}K_\eps (g_{\rm atom} - y_{\rm ref}), D_{\mathcal{R},\eps}\dot{u} ] \,d\tau \Big\rvert\\
&\quad + \Big\lvert \int_0^t (\dot{u}, \ddot{y} - \ddot{y}_{\rm ref})_{\ell^2_\eps} + \eps^d \sum_{x \in \sinto_\eps \Omega} A_\eps(x,\tau)[D_{\mathcal{R},\eps}(y - y_{\rm ref}), D_{\mathcal{R},\eps}\dot{u} ] \,d\tau \Big\rvert \Big)\\
&=: C \big( (C_g^2+ C_h^2) \eps^{2 \gamma} + \int_0^t \mathcal{E}(\tau) \,d\tau \big) + I_1 + I_2 + I_3 + I_4
\end{align*}
Clearly,
\[ I_2 \leq C\big(\int_0^t \mathcal{E}(\tau) \,d\tau + C_g^2 \eps^{2 \gamma}\big).\]
For $I_4$ we can use the estimates from the static case. Indeed, partial summation gives
\begin{align*}
I_4 &\leq C\Big( \int_0^t \mathcal{E}(t) \,d\tau + \int_0^t \lVert \ddot{y} - \ddot{y}_{\rm ref} - \divo_{\mathcal{R},\eps} (A_\eps (x, \tau) D_{\mathcal{R},\eps}(y-y_{\rm ref})) \rVert_{\ell^2_\eps(\into_\eps \Omega)}^2 \,d\tau \Big) \\
&= C\Big( \int_0^t \mathcal{E}(t) \,d\tau\\
&\quad + \int_0^t \lVert \ddot{y} - \ddot{y}_{\rm ref} - \divo_{\mathcal{R},\eps} (DW_{\rm atom} (D_{\mathcal{R},\eps}y) -DW_{\rm atom} (D_{\mathcal{R},\eps} y_{\rm ref})) \rVert_{\ell^2_\eps(\into_\eps \Omega)}^2 \,d\tau \Big)
\end{align*}
As we showed at the beginning of this proof, we have
\begin{align*}
\co \{D_{\mathcal{R},\eps} y_{\rm ref} (\hat{x}+ \eps \sigma), (\nabla y_{\rm ref} (x) \rho)_{\rho \in \mathcal{R}}\} \subset V
\end{align*}
for all $x\in \Omega_\eps$ and $\sigma \in \mathcal{R} \cup \{0\}$. We are therefore in a position to apply Proposition \ref{prop:ell2residuum}. 
\begin{align*}
\lVert \ddot{y} - &\ddot{y}_{\rm ref} - \divo_{\mathcal{R},\eps}(DW_{\rm atom}(D_{\mathcal{R},\eps} y)- DW_{\rm atom}(D_{\mathcal{R},\eps} y_{\rm ref})) \rVert_{\ell^2_\eps(\into_\eps \Omega)}\\
&= \lVert f_{\rm atom} - \ddot{y}_{\rm ref} + \divo_{\mathcal{R},\eps} DW_{\rm atom}(D_{\mathcal{R},\eps} y_{\rm ref}) \rVert_{\ell^2_\eps(\into_\eps \Omega)}\\
&\leq  \lVert f_{\rm atom} - \ddot{y}_{\rm ref} + \ddot{\tilde{y}}_{\rm cont} -\tilde{f} \rVert_{\ell^2_\eps(\into_\eps \Omega)}\\
&\quad + \lVert -\ddot{\tilde{y}}_{\rm cont} +\tilde{f}+ \divo_{\mathcal{R},\eps} DW_{\rm atom}(D_{\mathcal{R},\eps} y_{\rm ref}) \rVert_{\ell^2_\eps(\into_\eps \Omega)}\\
&\leq  \lVert f_{\rm atom} - f_{\rm ref} \rVert_{\ell^2_\eps(\into_\eps \Omega)} + \lVert -\ddot{y}_{\rm cont} +f+ \divo DW_{\rm CB}(\nabla y_{\rm ref}) \rVert_{L^2(\Omega_\eps; \R^d)}\\
&\quad +  C \eps^2 \Big\lVert \lVert \nabla^4 y_{\rm ref} \rVert_{L^\infty(B_{\eps R}(x))} + \lVert \nabla^3 y_{\rm ref} \rVert_{L^\infty(B_{\eps R}(x))}^\frac{3}{2} + \lVert \nabla^2 y_{\rm ref} \rVert_{L^\infty(B_{\eps R}(x))}^3 \\
&\quad +\eps\lVert \nabla^3 y_{\rm ref} \rVert_{L^\infty(B_{\eps R}(x))}^2 \Big\rVert_{L^2(\Omega_\eps)},
\end{align*}
where $C$ and $R$ just depend on $d, \mathcal{R}$ and $\lVert D^2 W_{\rm atom} \rVert_{C^2(V)}$. Now, remember that $y_{\rm ref} = \eta_\eps \ast (E y_{\rm cont})$. Hence, we can apply Proposition \ref{prop:approximation1} and Proposition \ref{prop:approximation2} and get 
\begin{align*}
\lVert \ddot{y} - &\ddot{y}_{\rm ref} - \divo_{\mathcal{R},\eps}(DW_{\rm atom}(D_{\mathcal{R},\eps} y)- DW_{\rm atom}(D_{\mathcal{R},\eps} y_{\rm ref})) \rVert_{\ell^2_\eps(\into_\eps \Omega)}\\
&\leq  \lVert f_{\rm atom} - f_{\rm ref} \rVert_{\ell^2_\eps(\into_\eps \Omega)}\\
&\quad + C \eps^2 \big(  \lVert \nabla^4 Ey_{\rm cont}\rVert_{L^2(\Omega_\eps + B_{(R+1)\eps}(0))} + \lVert \nabla^3 Ey_{\rm cont}\rVert_{L^3(\Omega_\eps + B_{(R+1)\eps}(0))}^{\frac{3}{2}}\\
&\quad + \lVert \nabla^2 Ey_{\rm cont}\rVert_{L^6(\Omega_\eps + B_{(R+1)\eps}(0))}^3 + \eps \lVert \nabla^3 Ey_{\rm cont}\rVert_{L^4(\Omega_\eps + B_{(R+1)\eps}(0))}^2 \\
&\quad + \lVert \nabla^2 Ey_{\rm cont} \rVert_{L^4(\Omega+B_{\eps}(0))} \lVert \nabla^3 Ey_{\rm cont} \rVert_{L^4(\Omega+B_{\eps}(0))} + \lVert \nabla^4 Ey_{\rm cont} \rVert_{L^2(\Omega+B_{\eps}(0))} \big)\\
&\leq \lVert f_{\rm atom} - f_{\rm ref} \rVert_{\ell^2_\eps(\into_\eps \Omega)} +C\eps^2 \lVert y_{\rm cont} \rVert_{H^4(\Omega; \R^d)} (1 + \lVert y_{\rm cont} \rVert_{H^4(\Omega; \R^d)}^2),
\end{align*}
where in the last step we used standard embedding theorems with $d \in \{2,3\}$, as well as the fact that $E$ is a continuous extension operator on all Sobolev spaces. Hence, we find
\begin{align*}
I_4 \leq C\big(\int_0^t \mathcal{E}(t) \,d\tau +   C_f^2 \eps^{2\gamma} + \eps^4 \big).
\end{align*}

Now let us look at the nonlinearity $I_1$. Evaluating the time derivative, we see that we can control it in terms of (some power of) the energy. But the resulting estimates are not good enough in $\eps$. The idea is to improve the estimates with a specific scheme of partial integrations in time. Indeed, it turns out that the estimates improve by $\eps^{\gamma-\frac{d}{2}}$ with each step. For this let us extend the definition of $A_\eps = A_{\eps,2}$ to
\[A_{\eps,k} = \int_0^1 D^k W_{\rm atom}\big(D_{\mathcal{R},\eps}y_{\rm ref} + s(D_{\mathcal{R},\eps}y -D_{\mathcal{R},\eps}y_{\rm ref} )\big)\,ds.
\]
Furthermore, let us write for $k \geq 2$
\begin{align*}
B_k(t) &= \int_0^1 \eps^d \sum_{x \in \sinto_\eps \Omega} \dot{A}_{\eps,k-1}[D_{\mathcal{R},\eps}u]^{k-1} s^{k-3} \,ds,\\
C_k(t) &= \int_0^1\eps^d \sum_{x \in \sinto_\eps \Omega} A_{\eps,k}[D_{\mathcal{R},\eps}u]^{k-1}[D_{\mathcal{R},\eps} \dot{y}_{\rm ref}] s^{k-3} \,ds,\\
D_k(t) &= \int_0^1\eps^d \sum_{x \in \sinto_\eps \Omega} A_{\eps,k}[D_{\mathcal{R},\eps}u]^{k-1}[D_{\mathcal{R},\eps} \dot{u}] s^{k-2} \,ds,\\
E_k(t) &= \int_0^1\eps^d \sum_{x \in \sinto_\eps \Omega} A_{\eps,k}[D_{\mathcal{R},\eps}u]^{k-1}[D_{\mathcal{R},\eps} (K_\eps (g_{\rm atom}-y_{\rm ref}))^\cdot] s^{k-2} \,ds,\\
F_k(t) &= \int_0^1\eps^d \sum_{x \in \sinto_\eps \Omega} A_{\eps,k}[D_{\mathcal{R},\eps}u]^{k} s^{k-2} \,ds.
\end{align*}
In this notation, we have
\[I_1 = C \Big\lvert \int_0^t B_3(\tau) \,d\tau \Big\rvert\] 
and, for $3 \leq k \leq m+1$, by partial integration in time,
\[\int_0^t B_{k}(\tau) + (k-1) D_{k-1}(\tau) \,d\tau =  F_{k-1}(t) - F_{k-1}(0), \]
as well as
\[B_k(t) = C_k(t) + D_k(t) + E_k(t)\]
by evaluating the time derivative. We claim to have relatively good estimates on the $C_k$, $E_k$, and $F_k$. At the same time we will prove estimates on the $D_k$ that get better with increasing $k$. Due to the two equations above this is sufficient. We just need to control all the $C_k$, $E_k$, and $F_k$, as well as $D_{m+1}$. Since
\[\lvert D_{\mathcal{R},\eps} u \rvert \leq \eps^{-\frac{d}{2}} \lVert u \rVert_{h^1_\eps(\sinto_\eps \Omega)} \leq B \eps^{\gamma-\frac{d}{2}},\]
we have the following estimates:
\begin{align*}
\Big\lvert \int_0^t C_k(\tau)\,d\tau \Big\rvert &\leq C (B \eps^{\gamma - \frac{d}{2}})^{k-3} \lVert y_{\rm cont} \rVert_{C^1(0,t; H^3(\Omega))}\int_0^t \mathcal{E}(\tau) \,d\tau, \\
\Big\lvert \int_0^t E_k(\tau)\,d\tau \Big\rvert &\leq C (B \eps^{\gamma - \frac{d}{2}})^{k-2} \Big(\int_0^t \mathcal{E}(\tau) \,d\tau + C_g^2 \eps^{2\gamma} \Big),  \\
\lvert F_k(0) \rvert &\leq C (B \eps^{\gamma - \frac{d}{2}})^{k-2} (C_g^2 + C_h^2) \eps^{2 \gamma},\\
\lvert F_k(t) \rvert &\leq C (B \eps^{\gamma - \frac{d}{2}})^{k-2} \mathcal{E}(t).
\end{align*}
Furthermore,
\begin{align*}
\Big\lvert \int_0^t D_{m+1}(\tau)\,d\tau \Big\rvert &\leq C (B \eps^{\gamma - \frac{d}{2}})^{m-1} \int_0^t \lVert u \rVert_{h^1_\eps(\sinto_\eps \Omega)} \lVert \dot{u} \rVert_{h^1_\eps(\sinto_\eps \Omega)} \,d\tau \\
&\leq C \eps^{-1} (B \eps^{\gamma - \frac{d}{2}})^{m-1} \int_0^t \lVert u \rVert_{h^1_\eps(\sinto_\eps \Omega)} \lVert \dot{u} \rVert_{\ell^2_\eps(\into_\eps \Omega)} \,d\tau \\
&\leq C \eps^{-1} (B \eps^{\gamma - \frac{d}{2}})^{m-1} \int_0^t \lVert u \rVert_{h^1_\eps(\sinto_\eps \Omega)}^2 + \lVert \dot{u} \rVert_{\ell^2_\eps(\into_\eps \Omega)}^2 \,d\tau \\
&\leq C \eps^{-1} (B \eps^{\gamma - \frac{d}{2}})^{m-1} \int_0^t \mathcal{E}(\tau) \,d\tau.
\end{align*}
Choosing $\eps_0$ small enough, such that $B \eps^{\gamma-\frac{d}{2}} \leq 1$, we can combine these estimates from $k=3$ up to $k=m+1$ to get
\[I_1 \leq C \Big( \int_0^t \mathcal{E}(\tau) \,d\tau + (C_g^2 + C_h^2)\eps^{2 \gamma} + B \eps^{\gamma-\frac{d}{2}} \mathcal{E}(t) + B^{m-1} \eps^{(m-1)(\gamma - \frac{d}{2})-1}\int_0^t \mathcal{E}(\tau) \,d\tau\Big)\]
for some $C=C(y_{\rm cont}, W_{\rm atom}, V, \mathcal{R}, \Omega, m)$. Choosing $\eps_0$ even smaller, we can ensure that $C B \eps^{\gamma-\frac{d}{2}} \leq \frac{1}{3}$ and $B^{m-1} \eps^{(m-1)(\gamma - \frac{d}{2})-1} \leq 1$, since $\gamma > \frac{d}{2} + \frac{1}{m-1}$ by assumption. Therefore,
\[I_1 \leq \frac{1}{3} \mathcal{E}(t) + C \Big( \int_0^t \mathcal{E}(\tau) \,d\tau + (C_g^2 + C_h^2)\eps^{2 \gamma}\Big).\]
The additional error term $I_3$ coming from the boundary conditions can be handled in a similar way. We now set
\begin{align*}
B_k(t) &= \int_0^1 \eps^d \sum_{x \in \sinto_\eps \Omega} \dot{A}_{\eps,k-1}[D_{\mathcal{R},\eps} K_\eps (g_{\rm atom}-y_{\rm ref})][D_{\mathcal{R},\eps}u]^{k-2} s^{k-3} \,ds,\\
C_k(t) &= \int_0^1\eps^d \sum_{x \in \sinto_\eps \Omega} A_{\eps,k}[D_{\mathcal{R},\eps} K_\eps (g_{\rm atom}-y_{\rm ref})][D_{\mathcal{R},\eps}u]^{k-2}[D_{\mathcal{R},\eps} \dot{y}_{\rm ref}] s^{k-3} \,ds,\\
D_k(t) &= \int_0^1\eps^d \sum_{x \in \sinto_\eps \Omega} A_{\eps,k}[D_{\mathcal{R},\eps} K_\eps (g_{\rm atom}-y_{\rm ref})][D_{\mathcal{R},\eps} u]^{k-2}[D_{\mathcal{R},\eps} \dot{u}] s^{k-2} \,ds,\\
E_k(t) &= \int_0^1\eps^d \sum_{x \in \sinto_\eps \Omega} A_{\eps,k}[D_{\mathcal{R},\eps} K_\eps (g_{\rm atom}-y_{\rm ref})][D_{\mathcal{R},\eps}u]^{k-2}\\
&\qquad [D_{\mathcal{R},\eps} (K_\eps (g_{\rm atom}-y_{\rm ref}))^\cdot] s^{k-2} \,ds,\\
F_k(t) &= \int_0^1\eps^d \sum_{x \in \sinto_\eps \Omega} A_{\eps,k}[D_{\mathcal{R},\eps} K_\eps (g_{\rm atom}-y_{\rm ref})][D_{\mathcal{R},\eps}u]^{k-1} s^{k-2} \,ds\\
G_k(t) &= \int_0^1\eps^d \sum_{x \in \sinto_\eps \Omega} A_{\eps,k}[D_{\mathcal{R},\eps} (K_\eps (g_{\rm atom}-y_{\rm ref}))^\cdot][D_{\mathcal{R},\eps}u]^{k-1} s^{k-2} \,ds.
\end{align*}
In analogy to before, we have
\[I_3 = C \Big\lvert \int_0^t D_2(\tau) \,d\tau \Big\rvert\] 
and, for $3 \leq k \leq m+1$,
\[\int_0^t B_{k}(\tau) + (k-2) D_{k-1}(\tau) + G_{k-1}(\tau) \,d\tau =  F_{k-1}(t) - F_{k-1}(0), \]
as well as
\[B_k(t) = C_k(t) + D_k(t) + E_k(t).\]
Again, we have the estimates
\begin{align*}
\Big\lvert \int_0^t C_k(\tau)\,d\tau \Big\rvert &\leq C (B \eps^{\gamma - \frac{d}{2}})^{k-3} \lVert y_{\rm cont} \rVert_{C^1(0,t; H^3(\Omega))}\Big(\int_0^t \mathcal{E}(\tau) \,d\tau + C_g^2 \eps^{2 \gamma} \Big), \\
\Big\lvert \int_0^t E_k(\tau)\,d\tau \Big\rvert &\leq C (B \eps^{\gamma - \frac{d}{2}})^{k-2} C_g^2 \eps^{2\gamma},  \\
\Big\lvert \int_0^t G_k(\tau)\,d\tau \Big\rvert &\leq C (B \eps^{\gamma - \frac{d}{2}})^{k-2} \Big(\int_0^t \mathcal{E}(\tau) \,d\tau + C_g^2 \eps^{2 \gamma} \Big),  \\
\lvert F_k(0) \rvert &\leq C (B \eps^{\gamma - \frac{d}{2}})^{k-2} (C_g^2 + C_h^2) \eps^{2 \gamma},\\
\lvert F_k(t) \rvert &\leq C (B \eps^{\gamma - \frac{d}{2}})^{k-2} (\mathcal{E}(t) +C_g^2 \eps^{2 \gamma}).
\end{align*}
Furthermore,
\begin{align*}
\Big\lvert \int_0^t D_{m+1}(\tau)\,d\tau \Big\rvert &\leq C \eps^{-1} (B \eps^{\gamma - \frac{d}{2}})^{m-1} \Big( \int_0^t \mathcal{E}(\tau) \,d\tau +C_g^2 \eps^{2 \gamma} \Big).
\end{align*}
As before this implies
\[I_3 \leq \frac{1}{3} \mathcal{E}(t) + C \Big( \int_0^t \mathcal{E}(\tau) \,d\tau + (C_g^2 + C_h^2)\eps^{2 \gamma}\Big)\]

Overall we proved
\begin{align*}
\mathcal{E}(t) &= 3 (\mathcal{E}(t) -  \frac{2}{3} \mathcal{E}(t))\\
&\leq  C \Big( (C_f^2 + C_g^2+ C_h^2 + \eps^{4-2\gamma}) \eps^{2 \gamma} + \int_0^t \mathcal{E}(\tau) \,d\tau\Big)
\end{align*}
for some $C=C(y_{\rm cont}, W_{\rm atom}, V, \mathcal{R}, \Omega, m, \gamma)$, all $\eps \leq \eps_0(y_{\rm cont}, W_{\rm atom}, V, \mathcal{R}, \Omega, m, \gamma, B)$ and $t \in [0, T_\eps)$.

Grönwall's inequality then yields
\begin{align*}
\mathcal{E}(t) &\leq C (C_f^2 + C_g^2+ C_h^2 + \eps^{4-2\gamma}) \eps^{2 \gamma} e^{Ct}\\
&\leq\frac{B^2}{2} \eps^{2 \gamma},
\end{align*}
where we have finally chosen $B:=\Big(2C (C_f^2 + C_g^2+ C_h^2 + 1) e^{CT_0}\Big)^{\frac{1}{2}}$. In particular, with $C\geq 1$ we satisfy the condition $B > \sqrt{2 C_g^2 + 2C_h^2}$, that we required at the beginning.

In particular, $\mathcal{E}(T_\eps) \leq \frac{B^2}{2} \eps^{2 \gamma}$ and therefore $T_\eps = T_0$ for $\eps \leq \eps_0$.

The convergence estimate immediately follows from the energy estimate and the estimate we assumed for the boundary conditions. Uniqueness follows directly from the local uniqueness in Theorem \ref{thm:discretelocalexistence}.
\end{proof}

\section*{Acknowledgments}
This work is part of my PhD thesis \cite{braunphdthesis}. I would like to express my gratitude to my advisor Bernd Schmidt for his guidance and support. Also, I am grateful that I had the opportunity to discuss different aspects of this work with Dirk Blömker, Christoph Ortner, and Florian Theil.

\appendix
\section{Additional Regularity of Weak Solutions to Second Order Hyperbolic Equations.}

\begin{lem} \label{lem:abscontofAuu}
Let $V$ be a real reflexive Banach space. Let $T>0$ and $A, A' \colon [0,T] \to L(V;V')$ be bounded such that for every $v_1, v_2 \in V$ the map $t \mapsto \langle A(t)v_1,v_2\rangle_{V',V}$ is absolutely continuous with derivative $\langle A'(t)v_1,v_2\rangle$. Furthermore, assume that $A(t)$ is symmetric. Let $u \in W^{1,1}(0,T;V)$. Then
\[ g \colon t \mapsto \langle A(t)u(t),u(t) \rangle_{V',V}\]
is absolutely continuous with derivative
\[g'(t) = \langle A'(t)u(t),u(t) \rangle_{V',V} + 2\langle A(t)u(t),u'(t) \rangle_{V',V}.\]
\end{lem}
\begin{proof}
Let $0\leq t_0 < t_1 \leq T$. For $k\in\N$ we write $\eps = \frac{t_1-t_0}{k}$ and for any map $f$ defined on $[t_0,t_1]$ we set
\begin{align*}
\tilde{f}_k(t) &= f\big(t_0 +(l+\frac{1}{2})\eps \big) &\text{ if } t \in t_0 + [ l \eps,(l+1) \eps) \\
f_k^+(t) &= f\big(t_0 + (l+\frac{1}{2})\eps\big) &\text{ if } t \in t_0 + [(l-\frac{1}{2}\eps,(l+\frac{1}{2}) \eps)  \\
f_k^\circ(t) &= f\big(t_0 + l\eps\big) &\text{ if } t \in t_0 + [(l-\frac{1}{2}\eps,(l+\frac{1}{2}) \eps)  \\
f_k^-(t) &= f\big(t_0 + (l-\frac{1}{2}) \eps \big) &\text{ if } t \in t_0 + [(l-\frac{1}{2}\eps,(l+\frac{1}{2}) \eps) 
\end{align*}
whenever $t \in t_0 + [l\eps,(l+1) \eps)$.
With this notation one easily calculates
\begin{align*}
&\int_{t_0}^{t_1} \langle A'(t)\tilde{u}_k(t),\tilde{u}_k(t) \rangle_{V',V} \,dt = -2 \int_{t_0 + \frac{\eps}{2}}^{t_1-\frac{\eps}{2}} \langle A^\circ_k(t)u(t),u'(t) \rangle_{V',V} \,dt\\
&\qquad\langle A(t_1)u^-_k(t_1),u^-_k(t_1) \rangle_{V',V} - \langle A(t_0)u^+_k(t_0),u^+_k(t_0) \rangle_{V',V} \\
&\qquad -2 \int_{t_0 + \frac{\eps}{2}}^{t_1-\frac{\eps}{2}} \langle A^\circ_k(t)\big(\frac{u_k^+(t)+ u_k^-(t)}{2} -u(t)\big),u'(t) \rangle_{V',V} \,dt.
\end{align*}
Since $\lVert u' \rVert \in L^1(0,T)$, standard theory of absolutely continuous functions gives
\[\big\lVert \frac{u_k^+ + u_k^-}{2} -v\big\rVert_{L^\infty(t_0+\frac{\eps}{2},t_1-\frac{\eps}{2};V)} \to 0 \]
as $k \to \infty$. Since also $A^\circ_k$ is uniformly bounded, the last integral goes to $0$ as $k \to \infty$. The middle terms converge to $\langle A(t_1)u(t_1),u(t_1) \rangle_{V',V} - \langle A(t_0)u(t_0),u(t_0) \rangle_{V',V}$ due to the continuity of $u$. The integrand in the first term of the right hand side (extended by $0$ if necessary) converges pointwise, since $s \mapsto \langle A(s)u(t),u'(t)\rangle_{V',V}$ is continuous. Therefore the limit function is measurable and the integral converges by Lebesgue's theorem to
\[\int_{t_0}^{t_1} \langle A(t)u(t),u'(t) \rangle_{V',V} \,dt.\]
The left hand side, also converges by Lebesgue's theorem. In particular, the map $t \mapsto \langle A(t)u(t),u'(t)\rangle_{V',V}$ is measurable and
\begin{align*}
&\int_{t_0}^{t_1} \langle A(t)u(t),u'(t)\rangle_{V',V} + 2 \langle A(t)u(t),u'(t) \rangle_{V',V} \,dt \\
&\quad= \langle A(t_1)u(t_1),u(t_1) \rangle_{V',V} - \langle A(t_0)u(t_0),u(t_0) \rangle_{V',V}
\end{align*} 
\end{proof}

\begin{thm} \label{thm:additionalregularity}
Let $V \hookrightarrow H \hookrightarrow V'$ be a real Gelfand triple with a reflexive Banach space $V$. Let $T>0$ and $A, A' \colon [0,T] \to L(V;V')$ be bounded such that for every $v_1, v_2 \in V$ the map $t \mapsto \langle A(t)v_1,v_2\rangle_{V',V}$ is absolutely continuous with derivative $\langle A'(t)v_1,v_2\rangle$. Furthermore, assume that $A(t)$ is symmetric and satisfies the uniform Gårding inequality
\[\langle A(t)v,v\rangle_{V',V} \geq c_1 \lVert v \rVert^2_V -c_2 \lVert v \rVert^2_H \]
for some constants $c_1>0$, $c_2 \in \R$ and all $t,v$. Let $u \in L^\infty(0,T;V)$ such that the weak derivatives $u' \in L^\infty(0,T;H)$ and $u'' \in L^1(0,T;V')$ exist and we have $u''+Au =:F \in L^1(0,T;H)$.

Then $u \in C([0,T];V) \cap C^1([0,T];H)$.
\end{thm}
\begin{proof}
The basic result goes back to \cite{strauss66} and we follow their ideas. Note though, that we have weaker assumptions on $A$ and $A'$. We also want to fix some small flaws in their proof.

Our main claim is that the energy
\[E(t)= \lVert u'(t) \rVert_H^2 + \langle A(t)u(t),u(t)\rangle_{V',V} \]
is continuous in $[0,T]$ even though $u'$ does not necessarily take values in $V$ as in Lemma \ref{lem:abscontofAuu}. For the moment, let us assume this claim is true and show that it proves the theorem.

Let $t_n \to t$. Clearly, $u \in C([0,T];H)\cap C^1([0,T];V')$. Therefore, $u(t_n) \to u(t)$ in $H$ and $u'(t_n) \to u'(t)$ in $V'$. Since we also have $u \in L^\infty(0,T;V)$ and $u' \in L^\infty(0,T;H)$ and since $V$ is reflexive, we get $u(t_n) \wto u(t)$ in $V$ and $u'(t_n) \wto u'(t)$ in $H$. The continuity of the energy now ensures that this weak convergence is actually strong convergence. In more detail,
\begin{align*}
&\lVert u'(t)-u'(t_n) \rVert_H^2 + \langle A(t) (u(t)-u(t_n)),u(t)-u(t_n)\rangle_{V',V}\\
&\qquad = E(t) + E(t_n) - 2 (u'(t),u'(t_n))_H - 2 \langle A(t)u(t),u(t_n)\rangle_{V',V}\\
&\qquad \to 0.
\end{align*}
Therefore, $u'(t_n) \to u'(t)$ in $H$ and, using the Gårding inequality, $u(t_n) \to u(t)$ in $V$.

Now we come to the main part of the proof. We have to show that the energy is continuous in $[0,T]$. Actually, we will even show that $E$ is absolutely continuous with
\[E'(t) = \langle A'(t)u(t),u(t) \rangle_{V',V} +2(F(t),u'(t))_H.\]
Fix $0\leq t_0<t_1\leq T$, let $\delta,\eps >0$ with $2\delta < t_1-t_0$ and define a continuous cutoff $\theta_\delta$ by setting $\theta_\delta(t)=1$ for $t \in [t_0 + \delta, t_1 - \delta]$, $\theta_\delta(t)=0$ for $t \notin [t_0 , t_1]$ and affine on each of the two remaining intervals. Let $\eta$ be the standard smoothing kernel and as always $\eta_\eps(t) = \eps^{-1} \eta(\frac{t}{\eps})$. Set $w_1:=\eta_\eps \ast (\theta_\delta u') \in C_c^\infty (\R;H)$ and $w_2:=\eta_\eps \ast (\theta_\delta u) \in C_c^\infty (\R;V)$. If we extend $A$ by $A(0)$ to the left and $A(T)$ to the right with $A'(t)=0$ outside of $[0,T]$, we can use Lemma \ref{lem:abscontofAuu} on some larger interval and get
\[0= \int_\R 2 (w_1(t),w_1'(t))_H + \langle A'(t) w_2(t),w_2(t)\rangle + 2 \langle A(t) w_2(t),w_2'(t)\rangle \,dt.\]
That is,
\begin{align*}
0&= \int_\R 2 (\eta_\eps \ast (\theta_\delta u'),\eta_\eps' \ast (\theta_\delta u'))_H + \langle A' \eta_\eps \ast (\theta_\delta u),\eta_\eps \ast (\theta_\delta u)\rangle_{V',V}\\
&\quad +2 \langle A \eta_\eps \ast (\theta_\delta u),\eta_\eps' \ast (\theta_\delta u))\rangle_{V',V} \,dt.
\end{align*}
Since
\[\eta_\eps' \ast (\theta_\delta u') = \eta_\eps \ast (\theta_\delta' u') + \eta_\eps \ast (\theta_\delta u'') \]
pointwise in $V'$, we see that $\eta_\eps \ast (\theta_\delta u'')$ actually takes values in $H$ and is even bounded in $L^\infty(\R;H)$ for fixed $\eps$ and varying $\delta$, since $\theta_\delta$ and $\theta_\delta'$ are bounded in $L^1(\R)$. Similarly,
\[\eta_\eps' \ast (\theta_\delta u) = \eta_\eps \ast (\theta_\delta' u) + \eta_\eps \ast (\theta_\delta u') \]
and $\eta_\eps \ast (\theta_\delta u')$ is bounded in $L^\infty(\R;V)$ for fixed $\eps$ and varying $\delta$.

Thus, we can rewrite the above equality and get
\begin{align*}
0&= \int_\R 2 (\eta_\eps \ast (\theta_\delta u'),\eta_\eps \ast (\theta_\delta' u'))_H + 2 \langle \eta_\eps \ast (\theta_\delta u''), \eta_\eps \ast (\theta_\delta u')\rangle_{V',V}\\
&\quad + \langle A' (\eta_\eps \ast (\theta_\delta u)),\eta_\eps \ast (\theta_\delta u)\rangle_{V',V} + 2 \langle A (\eta_\eps \ast (\theta_\delta u)) -\eta_\eps \ast (\theta_\delta Au),\eta_\eps' \ast (\theta_\delta u))\rangle_{V',V}\\
&\quad +2 \langle \eta_\eps \ast (\theta_\delta Au),\eta_\eps \ast (\theta_\delta' u))\rangle_{V',V} +2 \langle \eta_\eps \ast (\theta_\delta Au),\eta_\eps \ast (\theta_\delta u'))\rangle_{V',V} \,dt.
\end{align*}
Now we want to let $\delta \to 0$. Since $\theta_\delta \to \chi_{[t_0, t_1]}$ in the $(L^\infty,L^1)$-Mackey topology, we have
\begin{align*}
\eta_\eps \ast (\theta_\delta u) &\to \eta_\eps \ast (\chi_{[t_0,t_1]} u)&\text{in } L^\infty(\R;V)\\
\eta_\eps \ast (\theta_\delta u') &\to \eta_\eps \ast (\chi_{[t_0,t_1]} u')&\text{in } L^\infty(\R;H)\\
\eta_\eps \ast (\theta_\delta F) &\to \eta_\eps \ast (\chi_{[t_0,t_1]} F)&\text{in } L^\infty(\R;H)\\
\eta_\eps \ast (\theta_\delta Au) &\to \eta_\eps \ast (\chi_{[t_0,t_1]} Au)&\text{in } L^\infty(\R;V')\\
\eta_\eps' \ast (\theta_\delta u) &\to \eta_\eps' \ast (\chi_{[t_0,t_1]} u)&\text{in } L^\infty(\R;V).
\end{align*}
Therefore, we have
\begin{align*}
0&= \lim_{\delta \to 0}\int_\R 2 (\eta_\eps \ast \eta_\eps \ast (\chi_{[t_0,t_1]} u'),\theta_\delta' u')_H + 2 ( \eta_\eps \ast (\chi_{[t_0,t_1]} F), \eta_\eps \ast (\chi_{[t_0,t_1]} u'))_H\\
&\quad + \langle A' (\eta_\eps \ast (\chi_{[t_0,t_1]} u)),\eta_\eps \ast (\chi_{[t_0,t_1]} u)\rangle_{V',V}\\
&\quad + 2 \langle A (\eta_\eps \ast (\chi_{[t_0,t_1]} u)) -\eta_\eps \ast (\chi_{[t_0,t_1]} Au),\eta_\eps' \ast (\chi_{[t_0,t_1]} u))\rangle_{V',V}\\
&\quad +2 \langle \eta_\eps \ast \eta_\eps \ast (\chi_{[t_0,t_1]} Au),\theta_\delta' u)\rangle_{V',V}\,dt.
\end{align*}
The maps
\[t \mapsto (\eta_\eps \ast \eta_\eps \ast (\chi_{[t_0,t_1]} u')(t),u'(t))_H\]
and
\[t \mapsto \langle \eta_\eps \ast \eta_\eps \ast (\chi_{[t_0,t_1]} Au)(t),u(t)\rangle_{V',V}\]
are continuous in $[t_0,t_1]$ as a product of a continuous and a weakly continuous function and as a product of a continuous and a weak-$\ast$ continuous function, respectively. Plugging in the actual values of $\theta_\delta'$, we therefore get
\begin{align*}
0&= 2(\eta_\eps \ast \eta_\eps \ast (\chi_{[t_0,t_1]} u')(t_0),u'(t_0))_H - 2(\eta_\eps \ast \eta_\eps \ast (\chi_{[t_0,t_1]} u')(t_1),u'(t_1))_H\\
&\quad+ 2\langle \eta_\eps \ast \eta_\eps \ast (\chi_{[t_0,t_1]} Au)(t_0),u(t_0)\rangle_{V',V} - 2\langle \eta_\eps \ast \eta_\eps \ast (\chi_{[t_0,t_1]} Au)(t_1),u(t_1)\rangle_{V',V}\\
&\quad +\int_\R 2 ( \eta_\eps \ast (\chi_{[t_0,t_1]} F), \eta_\eps \ast (\chi_{[t_0,t_1]} u'))_H\\
&\quad + \langle A' (\eta_\eps \ast (\chi_{[t_0,t_1]} u)),\eta_\eps \ast (\chi_{[t_0,t_1]} u)\rangle_{V',V}\\
&\quad + 2 \langle A (\eta_\eps \ast (\chi_{[t_0,t_1]} u)) -\eta_\eps \ast (\chi_{[t_0,t_1]} Au),\eta_\eps' \ast (\chi_{[t_0,t_1]} u))\rangle_{V',V}\,dt.
\end{align*}
Next, we want to send $\eps \to 0$. For any $v \in C^1_c(\R;V)$, we have $\eps \eta_\eps' \ast v \to 0$ in $L^2(\R;V)$. Therefore, for any such $v$,
\[\limsup_{\eps \to 0} \lVert \eps \eta_\eps' \ast (\chi_{[t_0,t_1]} u) \rVert_{L^2(\R;V)} \leq \lVert \eta' \rVert_{L^1(\R)} \lVert \chi_{[t_0,t_1]} u - v\rVert_{L^2(\R;V)}.\]
Hence, $\eps \eta_\eps' \ast (\chi_{[t_0,t_1]} u) \to 0$ in $L^2(\R;V)$. At the same time
\[\lVert A (\eta_\eps \ast (\chi_{[t_0,t_1]} u)) -\eta_\eps \ast (\chi_{[t_0,t_1]} Au) \rVert_{L^2(\R; V')} \leq \eps \lVert A' \rVert_{L^\infty} \lVert \eta_\eps \ast (\chi_{[t_0,t_1]} u) \rVert_{L^2(\R; V')}\]
and so the last term in the integral goes to $0$.

Now $\eta_\eps \ast \eta_\eps \geq 0$ with $\int_0^\infty \eta_\eps \ast \eta_\eps \,dt = \frac{1}{2}$. Therefore, by Lebesgue's theorem and the weak continuity of $u'$ in $H$,
\begin{align*}
(\eta_\eps \ast \eta_\eps \ast (\chi_{[t_0,t_1]} u')(t_0),u'(t_0))_H &= \int_0^\infty (u'(t_0+\eps s),u'(t_0))_H\ \eta \ast \eta (s)\,ds\\
&\to \int_0^\infty (u'(t_0),u'(t_0))_H\ \eta \ast \eta (s)\,ds\\
&= \frac{1}{2} (u'(t_0),u'(t_0))_H .
\end{align*}
Similar results hold for the other terms and we conclude that
\begin{align*}
E(t_1) -E(t_0) = \int_{t_0}^{t_1} 2(F,u')_H + \langle A'u,u \rangle_{V',V} \,dt
\end{align*}
for any $0 \leq t_0<t_1 \leq T$.
\end{proof}

\section{Multiplication of Sobolev Functions}

The following Lemma is very useful to control products of Sobolev functions with the same integrability exponent $p$.

\begin{lem} \label{lem:prodofsobolev}
Let $1\leq p<\infty$ and $K,M \in \N$, such that $K \geq 1$ and $M  > \frac{d}{p}$and let $\lambda_k \in \N_0$ for $1 \leq k \leq K$ with $\sum_{k=1}^K \lambda_k =:N \leq M$. Let $\Omega \subset \R^d$ be open and bounded with $\partial \Omega$ Lipschitz. Then there is a $C>0$ such that for any $f_k \in W^{M-\lambda_k,p}$ we have $\prod_{k=1}^K f_k \in W^{M-N,p}(\Omega)$
with
\[\big\lVert \prod_{k=1}^K f_k \big\rVert_{W^{M-N,p}(\Omega)} \leq C \prod_{k=1}^K \lVert f_k \rVert_{W^{M-\lambda_k,p}(\Omega)}. \]
Additionally, the product mapping is continuous even from the weak topologies on the $W^{M-\lambda_k,p}(\Omega)$ to the strong topology on $W^{L,p}(\Omega)$ if either $L < M - N$ or $L=M - N$ and $\lambda_k < N$ for all $k$.

In particular, $W^{M,p}(\Omega)$ is a Banach algebra.
\end{lem}
\begin{proof}
By density, it suffices to consider functions in $C^{\infty}(\overline{\Omega})$. Furthermore, by the product rule, it suffices to prove $\prod_{k=1}^K f_k \in L^p(\Omega)$
with
\[\big\lVert \prod_{k=1}^K f_k \big\rVert_{L^p(\Omega)} \leq C \prod_{k=1}^K \lVert f_k \rVert_{W^{M-\lambda_k,p}(\Omega)}. \]
Standard embedding theorems give $W^{M-l,p} \hookrightarrow L^\infty$ for $M-l > \frac{d}{p}$,  $W^{M-\lambda,p} \hookrightarrow L^q$ for any $1 \leq q < \infty$ and $M-l = \frac{d}{p}$ and $W^{M-l,p} \hookrightarrow L^q$ for $M-l < \frac{d}{p}$ and $q \leq \frac{d}{\frac{d}{p}-(M-l)}$. Let us write
\begin{align*}
A&= \{k \colon \lambda_k > M - \frac{d}{p}\} \\
B& = \{k \colon \lambda_k = M - \frac{d}{p}\}.
\end{align*}
Hölder's inequality now establishes the claim if either $B= \emptyset$ and $\sum_{k \in A} \frac{\frac{d}{p}-M+\lambda_k}{\frac{d}{p}} \leq 1$ or $B \neq \emptyset$ and $\sum_{k \in A} \frac{\frac{d}{p}-M+\lambda_k}{\frac{d}{p}} < 1$. This is trivially true if $A = \emptyset$ or if $\lvert A \rvert = 1$ and $B=\emptyset$. If now $\lvert A \rvert = 1$ but $B \neq \emptyset$, writing $A = \{k_0\}$, we find that $\lambda_{k_0} \leq N - (M-\frac{d}{2}) < N \leq M$ and the condition is satisfied. Finally, if $\lvert A \rvert \geq 2$ then
\[\sum_{k \in A} \frac{\frac{d}{2}-M+\lambda_k}{\frac{d}{2}} \leq  \frac{\lvert A \rvert (\frac{d}{2}-M) + N}{\frac{d}{2}} \leq \frac{2 (\frac{d}{2}-M) + M}{\frac{d}{2}} = 1 - \frac{M-\frac{d}{2}}{\frac{d}{2}} <1.\]
For the additional claim, note that, even after taking up to $L$ derivatives with the product rule, we always have $\lambda_k <M$, so that all the functions are embedded into some space of lower differentiability. The embeddings were already compact if $M- \lambda_k \geq \frac{d}{p}$. To control the other case we just have to make sure that we always have the strict inequality $\sum_{k \in A} \frac{\frac{d}{p}-M+\lambda_k}{\frac{d}{p}} < 1$. This was only unclear in the case $\lvert A \rvert = 1$, $B = \emptyset$. But since $\lambda_k < M$, we now have a strict inequality in this case too.
\end{proof}

We also need the following multivariate version of the Faà di Bruno formula:
\begin{lem} \label{lem:multivariateFaaDiBruno}
Let $n,d,k,l\in \N$. Let $\Omega \subset \R^d$ open, $g \in C^n(\Omega; \R^k)$ and $f \in C^n(\R^k;\R^l)$. Then $f \circ g \in C^n(\Omega; \R^l)$ with
\[D^\alpha (f \circ g)(x) = \sum\limits_{\substack{\beta \in \N_0^k \\ 1 \leq \lvert \beta \rvert \leq \lvert \alpha \rvert}} D^\beta f(g(x)) \sum_{s=1}^{\lvert \alpha \rvert}\sum\limits_{p_s(\alpha,\beta)} \alpha! \prod_{j=1}^s\frac{(D^{\gamma_j}g(x))^{\lambda_j}}{\lambda_j! (\gamma_j!)^{\lvert \lambda_j \rvert}} \]
for all $x \in \Omega$ and $\alpha \in \N_0^d$ with $1 \leq \lvert \alpha \rvert \leq n$, where
\begin{align*}
p_s(\alpha,\beta)&=\Big\{ (\lambda_1,\dots,\lambda_s;\gamma_1,\dots,\gamma_s) \colon \lambda_j \in \N_0^k, \gamma_j \in \N_0^d,\\
&\quad 0\prec \gamma_1 \prec \dots \prec \gamma_s, \lvert \lambda_j \rvert > 0, \sum_{j=1}^s \lambda_j = \beta, \sum_{j=1}^s \gamma_j \lvert \lambda_j \rvert = \alpha  \Big\}
\end{align*}
and $\gamma_1 \prec \gamma_2$ if and only if $\lvert \gamma_1 \rvert< \lvert \gamma_2 \rvert$ or $\lvert \gamma_1 \rvert = \lvert \gamma_2 \rvert$ and, for some $j$, 
\[(\gamma_1)_1 = (\gamma_2)_1, \dots, (\gamma_1)_{j-1} = (\gamma_2)_{j-1} \text{ and } (\gamma_1)_j < (\gamma_2)_j.\]
\end{lem}
\begin{proof}
Even though it is quite possible that the result itself is much older, it can be found in \cite{constantinesavits96}.
\end{proof}
As a corollary we get the following statements.
\begin{cor} \label{cor:FaadiBruno}
Let $n,d,k,l\in \N$. There are $C=C(n,d,k,l)>0$ such that the following holds. Let $\Omega \subset \R^d$ open, $g \in C^n(\Omega; \R^k)$ and $f \in C^n(\R^k;\R^l)$. Then
\[\lvert D^n(f \circ g) (x) \rvert \leq C \sum_{s=1}^n \lvert D^s f(g(x)) \rvert \sum_{\substack{l_1, \dots, l_s \geq 1 \\ l_1+ \dots+ l_s = n}} \prod_{j=1}^s \lvert D^{(l_j)}g(x)\rvert,\]
for all $x\in\Omega$. Furthermore, if $f$ and all its derivatives are uniformly continuous and $h \in C^n(\Omega; \R^k)$ then
\begin{align*}
&\lvert D^n(f \circ (g+h)) (x) - D^n(f \circ g) (x) \rvert \\
&\quad \leq C \sum_{s=1}^n \sum_{\substack{l_1, \dots, l_s \geq 1 \\ l_1+ \dots+ l_s = n}} \omega_{D^s f}(\lvert h(x) \rvert) \prod_{j=1}^s \lvert D^{l_j}g(x)\rvert \\
&\qquad+C \sum_{s=1}^n  \sum_{\substack{l_1, \dots, l_s \geq 1 \\ l_1+ \dots+ l_s = n}} \sum_{m=1}^s \lvert D^s f(g(x)+h(x)) \rvert \prod_{j=1}^m \lvert D^{l_j}h(x)\rvert \prod_{j=m+1}^s \lvert D^{l_j}g(x)\rvert
\end{align*} 
\end{cor}

\begin{lem} \label{lem:composition}
Let $m \in \N_0$, $d<2m+2$ and let $\Omega \subset \R^d$ be an open, bounded set with Lipschitz boundary. Let $V \subset \R^{d \times \mathcal{R}}$ be open and $W_{\rm atom} \in C^{m+2}(V)$.

Now define the operator $F\colon B \mapsto DW_{\rm CB} \circ B$. Then
\[\{B\in H^{m+1}(\Omega;\R^{d \times d}) \colon \inf\limits_{x \in \Omega} \dist ( (B(x) \rho)_{\rho \in \mathcal{R}}, V^c)>0\}\]
is open in $H^{m+1}(\Omega;\R^{d \times d})$ and
\[F\colon \{B\in H^{m+1}(\Omega;\R^{d \times d}) \colon \inf\limits_{x \in \Omega} \dist ( (B(x) \rho)_{\rho \in \mathcal{R}}, V^c)>0\} \to H^{m+1}(\Omega;\R^{d \times d})\]
is well-defined, continuous and bounded. Furthermore, if $W_{\rm atom} \in C^{m+3}(V)$, then $F$ is $C^1$ with
\[DF(B)[H](x) = D^2 W_{\rm CB} (B(x))[H(x)].\]
\end{lem}
\begin{proof}
This is for the most part contained in \cite[I.~Thm.3.1]{valentbvps} and \cite[II.~Thm.4.1]{valentbvps}. Only the boundedness is not explicitly mentioned, but it follows along the same lines.
\end{proof}

\section{Elliptic Regularity with Sobolev Coefficients} \label{app:ellipticregularity}
We need a result on higher order regularity for linear systems that are elliptic in the Legendre-Hadamard sense. To be useful for quasilinear equations it is crucial that the regularity assumptions on the coefficients are not too strong. In the standard literature the typical assumption for $W^{k+2,p}$-regularity of the solution is $A \in C^{k,1} = W^{k+1,\infty}$ or, more rarely, $A \in W^{k+1,p}$ if $p > d$. We will reduce the last assumption to the much weaker condition $p(k+1)>d$. It is no coincidence that this assumption corresponds to what is needed for $A$ to be continuous. Actually, this is known to be the critical case. It seems reasonable that the case $p(k+1)=d$ can be included, as there are regularity results where the coefficients are not continuous but only have vanishing mean oscillation, but we will not investigate this question here.

For the sake of generality we will consider the general case $1<p<\infty$ but we are mostly interested in the case $p=2$. Even though it seems quite possible that this kind of result has been proven before, it does not seem to be available in the standard literature. It is largely, but not quite, contained in \cite{simpsonspector09} and, of course, builds heavily on the famous classical work \cite{agmondouglisnirenbergII}.

We consider a differential operator in divergence form
\[(Lu)_i = (- \divo(A \nabla u) + b \nabla u)_i = -\sum_{j,k,l} \frac{\partial}{\partial x_j} (A_{ijkl} \frac{\partial u_k}{\partial x_l}) + \sum_{k,l} b_{ikl} \frac{\partial u_k}{\partial x_l}. \]
In particular, we are interested in the cases $b=0$ and, if $A$ is Lipschitz, $b_{ikl} = \sum_j \frac{\partial A_{ijkl}}{\partial x_j}$. The second case corresponds to an operator in non-divergence form
\[(Lu)_i = \sum_{j,k,l} A_{ijkl} \frac{\partial^2 u_k}{\partial x_j \partial x_l}.\]
On some open set $\Omega \subset \R^d$ define the corresponding bilinear form
\[B(u,v) = \int_\Omega  \sum_{i, j,k,l} A_{ijkl} \frac{\partial v_i}{\partial x_j} \frac{\partial u_k}{\partial x_l} + \sum_{i,k,l} b_{ikl} v_i \frac{\partial u_k}{\partial x_l} \,dx.\]
whenever it is well-defined.

Let us first recall the classical Gårding inequality:
\begin{thm} \label{thm:contGårding}
Let $\Omega \subset \R^d$ be an open and bounded set and $\lambda_0 > 0$ such that
\[\sum_{i,j,k,l} A_{ijkl}(x) \xi_i \eta_j \xi_k \eta_l \geq \lambda_0 \lvert \xi \rvert^2 \lvert \eta \rvert^2\]
for all $x\in \Omega,\xi \in \R^N,\eta \in \R^d$. Furthermore, assume that $A$ is bounded and uniformly continuous with modulus $\omega$ and $b \in L^\infty(\Omega)$.

Then there exists a $\lambda_1=\lambda_1(\lVert A \rVert_\infty, \lambda_0, \Omega, \omega, \lVert b \rVert_\infty) \geq 0$, such that
\[ \frac{\lambda_0}{2} \int_\Omega \lvert \nabla u \rvert^2 \,dx \leq B(u,u) + \lambda_1 \int_\Omega \lvert u \rvert^2 \,dx \]
for all $u \in H^1_0(\Omega; \R^N)$.

If $A$ is constant and $b=0$, we can take $\lambda_1=0$ and can even achieve $\lambda_0$ instead of $\frac{\lambda_0}{2}$ as the constant on the left side.
\end{thm}

We have the following a priori estimates:

\begin{thm} \label{thm:simpsonspectorapriori}
Let $k \in \N_0$ and let $\Omega \subset \R^d$ be open and bounded with $C^{k+2}$ boundary. Let $1<p<\infty$ and assume that $(k+1)p>d$. Let $\lambda, \Lambda >0$ and let $A \in W^{k+1,p}(\Omega;\R^{d\times d \times d \times d})$, such that $\lVert A \rVert_{W^{k+1,p}(\Omega)} \leq \Lambda$ and
\[A(x)[\xi \otimes \eta, \xi \otimes \eta] \geq \lambda \lvert \xi \rvert^2 \lvert \eta \rvert^2\]
for all $\xi, \eta \in \R^d$ and $x \in \Omega$. Then there is a $C=C(\Omega,\lambda, \Lambda, p, k) $ such that for all $r \in \{0, \dots, k\}$ and $u \in W^{r+2,p}(\Omega;\R^d)$ we have
\[ \lVert u \rVert_{W^{r+2,p}(\Omega)} \leq C (\lVert \divo(A \nabla u) \rVert_{W^{r,p}(\Omega)}+\lVert u \rVert_{L^p(\Omega)}).\]
We also have the estimate in non-divergence form
\[\lVert u \rVert_{W^{r+2,p}(\Omega)} \leq C \Big(\Big\lVert \Big( \sum_{j,k,l} A_{ijkl} \frac{\partial^2 u_k}{\partial x_j \partial x_l} \Big)_i \Big\rVert_{W^{r,p}(\Omega)}+\lVert u \rVert_{L^p(\Omega)}\Big).\]
\end{thm}
\begin{proof}
In \cite{simpsonspector09} this has been proven in the case $r=k$ in divergence form based on the estimates of \cite{agmondouglisnirenbergII} for constant coefficients. But all other cases follow mostly along the same lines. This includes the case in non-divergence form since the proof is based on approximating $A$ locally by a constant. The case of smaller $r$ follows along the same lines as well, since the Agmon-Douglis-Nirenberg estimates are still valid. The only difference in all these cases is that one can no longer use the tame estimate of Moser. Instead one has to use the following finer estimates for multiplications in Sobolev spaces, namely
\begin{align*}
\lVert J \partial_l u \rVert_{W^{r+1,p}(\Omega)} &\leq \eps \lVert J \rVert_{W^{k+1,p}(\Omega)} \lVert u \rVert_{W^{r+2,p}(\Omega)} + C_\eps \lVert J \rVert_{W^{k+1,p}(\Omega)} \lVert u \rVert_{L^p(\Omega)}\\
&\qquad + C \lVert J \rVert_{L^\infty(\Omega)} \lVert u \rVert_{W^{r+2,p}(\Omega)}
\end{align*}
and
\begin{align*}
\lVert J \partial_j \partial_l u \rVert_{W^{r,p}(\Omega)} &\leq \eps \lVert J \rVert_{W^{k+1,p}(\Omega)} \lVert u \rVert_{W^{r+2,p}(\Omega)} + C_\eps \lVert J \rVert_{W^{k+1,p}(\Omega)} \lVert u \rVert_{L^p(\Omega)}\\
&\qquad + C \lVert J \rVert_{L^\infty(\Omega)} \lVert u \rVert_{W^{r+2,p}(\Omega)}
\end{align*}
for all $J \in W^{k+1,p}(\Omega), u \in W^{r+2,p}(\Omega)$, $\eps >0$, $0 \leq r \leq k$, $j$ and $l$ with constants that may depend on $p$, $k$ and $\Omega$. Both inequalities can be proven rather easily using the product rule, the Sobolev and Rellich–Kondrachov embedding theorems, as well as Ehrling's lemma.
\end{proof}

\begin{thm} \label{thm:optimalsobolevregularity}
Let $k \in \N_0$ and let $\Omega \subset \R^d$ be open and bounded with $C^{k+2}$ boundary. Let $1<p<\infty$,  $r \in \{0, \dots, k\}$ and assume that $(k+1)p>d$ and $p \geq \frac{2d}{d+2(r+1)}$. Let $\lambda, \Lambda_1 ,\Lambda_2 >0$ and let $A \in W^{k+1,p}(\Omega;\R^{d\times d \times d \times d})$, such that $\lVert A \rVert_{W^{k+1,p}(\Omega)} \leq \Lambda_1$ and
\[A(x)[\xi \otimes \eta, \xi \otimes \eta] \geq \lambda \lvert \xi \rvert^2 \lvert \eta \rvert^2\]
for all $\xi, \eta \in \R^d$ and $x \in \Omega$. Consider
\[L_{1,\mu}, L_{2,\mu} \colon W^{r+2,p}(\Omega;\R^d) \cap H^1_0(\Omega;\R^d) \to W^{r,p}(\Omega;\R^d)\]
defined by
\begin{align*}
L_{1,\mu} u &= -\divo (A \nabla u) + \mu u,\\
(L_{2,\mu} u)_i &= -\sum_{j,k,l} A_{ijkl} \frac{\partial^2 u_k}{\partial x_j \partial x_l} + \mu u_i.
\end{align*}
There is a $\mu_1=\mu_1(\Omega,k,p,\Lambda_1, \lambda)$ such that for all $\mu \geq \mu_1$ $L_{1,\mu}$ is an isomorphism. If additionally we have $A \in W^{1,\infty}(\Omega;\R^{d\times d \times d \times d})$ with $\lVert A \rVert_{W^{1,\infty}} \leq \Lambda_2$, then there is a $\mu_2=\mu_2(\Omega, k, p, \Lambda_1, \Lambda_2, \lambda)$ such that for all $\mu \geq \mu_2$ $L_{2,\mu}$ is an isomorphism. Furthermore, we have the estimates
\[\lVert L_{i,\mu}^{-1}\rVert \leq C_i,\]
where $C_1=C_1(\Omega,\lambda,\Lambda_1,p,k,\mu_{\rm max})>0$ and $C_2=C_2(\Omega,\lambda,\Lambda_1, \Lambda_2, p, k, \mu_{\rm max})>0$ and $\mu_i \leq \mu \leq \mu_{\rm max}$.
\end{thm}
\begin{proof}
We argue by continuity. Set $(A^t)_{ijkl} = t A_{ijkl} +(1-t) \delta_{ik} \delta_{jl}$ for $t \in [0,1]$ and denote by $L_{1,\mu}^t, L_{2,\mu}^t$ the corresponding operators. Since $(k+1)p>d$, these operators are well defined by Lemma \ref{lem:prodofsobolev} with
\[ \lVert L_{i,\mu}^t \rVert_{L(W^{r+2,p}(\Omega), W^{r,p}(\Omega))} \leq 1+ \lvert \mu \rvert + C(\Omega,k,p) \lVert A \rVert_{W^{k+1,p}(\Omega)} \]
for all $t,i,\mu$.

\emph{Claim 1: There are $\mu_i$, $i=1,2$, such that $L_{i,\mu}^t$ is one-to-one for all $\mu \geq \mu_i$.}

We can apply Theorem \ref{thm:contGårding}. Note that the modulus $\omega$ can be chosen only dependent on $\Omega, k, p, \Lambda$. If $L_{1,\mu}^t u = 0$, we can apply Theorem \ref{thm:contGårding} with $b=0$ to obtain
\[\frac{\lambda}{2} \int_\Omega \lvert \nabla u \rvert^2 \,dx \leq 0\]
and thus $u=0$, whenever $\mu \geq \mu_1 = \mu_1(\Omega, k, p, \Lambda_1, \lambda)$. If $L_{2,\mu}^t u = 0$, we can apply Theorem \ref{thm:contGårding} with $b_{ikl} = -\sum_j \frac{\partial A_{ijkl}}{\partial x_j}$ to obtain
\[\frac{\lambda}{2} \int_\Omega \lvert \nabla u \rvert^2 \,dx \leq 0\]
and thus $u=0$, whenever $\mu \geq \mu_2 = \mu_2(\Omega, k, p, \Lambda_1, \Lambda_2, \lambda)$.

\emph{Claim 2: There are constants $C_1, C_2>0$ with $C_1=C_1(\Omega,\lambda,\Lambda_1,p,k,\mu_{\rm max})$ and $C_2=C_2(\Omega,\lambda,\Lambda_1, \Lambda_2, p, k, \mu_{\rm max})$ such that for all $t \in [0,1]$, $u \in W^{r+2,p}(\Omega; \R^d) \cap H^1_0(\Omega;\R^d)$, $i \in \{1,2\}$ and $\mu_i \leq \mu \leq \mu_{\rm max}$ we have
\[ \lVert u \rVert_{W^{r+2,p}(\Omega)} \leq C_i \lVert L_{i,\mu}^t u \rVert_{W^{r,p}(\Omega)}.\]}
We argue by contradiction. If there were no such $C$, then there exist $t_n, \mu_n, u_n, A_n$ such that
\[ 1=\lVert u_n \rVert_{W^{r+2,p}(\Omega)} > n \lVert L_{i,\mu_n}^{t_n}(A_n) u_n \rVert_{W^{r,p}(\Omega)}.\]
Furthermore,
\[A_n(x)[\xi \otimes \eta, \xi \otimes \eta] \geq \lambda \lvert \xi \rvert^2 \lvert \eta \rvert^2\]
for all $\xi, \eta \in \R^d$ and $x \in \Omega$ and $\lVert A_n \rVert_{W^{k+1,p}} \leq \Lambda_1$. Then we find a subsequence (not relabeled), such that $t_n \to t$, $\mu_n \to \mu$ and $u_n \wto u$ in $W^{r+2,p}$ and $A_n \wto A$ in $W^{k+1,p}$. In particular $A_n \to A$ uniformly, $A$ is still elliptic with constant $\lambda$ and $u_n \to u$ strongly in $W^{r+1,p}$ and weakly in $H^1_0$. If $i=1$ we easily deduce that $L_{1,\mu_n}^{t_n}(A_n) u_n \to L_{1,\mu}^{t}(A) u$ in distribution. If $i=2$ we additional have $\lVert A_n \rVert_{W^{1,\infty}} \leq \Lambda_2$. By uniform convergence we also have $\lVert A \rVert_{W^{1,\infty}} \leq \Lambda_2$. We also have directly $L_{2,\mu_n}^{t_n}(A_n) u_n \wto L_{2,\mu}^{t}(A) u$ in $L^p$. But in both cases we also now that $\lVert L_{i,\mu_n}^{t_n}(A_n) u_n \rVert_{W^{r,p}(\Omega)} \to 0$. Hence, $L_{i,\mu}^{t}(A) u = 0$ with $u \in W^{r+2,p}(\Omega;\R^d) \cap H^1_0(\Omega;\R^d)$ and thus $u=0$ by claim 1. Now we use Theorem \ref{thm:simpsonspectorapriori} to find
\[1 = \lVert u_n \rVert_{W^{r+2,p}(\Omega)} \leq C \Big(\lVert L_{i,\mu_n}^{t_n}(A_n) u_n \rVert_{W^{r,p}(\Omega)}+ \lvert \mu_n \rvert \lVert u_n \rVert_{W^{r,p}(\Omega)} + \lVert u_n \rVert_{L^p(\Omega)}\Big).\]
Since the right hand side goes to $0$, we have a contradiction.

\emph{Claim 3: For $\mu \geq \mu_i$ the sets
\[I_{i,\mu} = \{t \in [0,1] \colon L_{i,\mu}^{t} \text{is onto}\}\]
are closed in $[0,1]$.}

Given $t_n \to t$, $t_n \in I_{i,\mu}$ and $f \in W^{r,p}(\Omega; \R^d)$, there are $u_n \in W^{r+2,p}(\Omega; \R^d) \cap H^1_0(\Omega;\R^d)$ such that $L_{i,\mu}^{t_n} u_n =f$. By claim 2 the $u_n$ are bounded in $W^{r+2,p}(\Omega; \R^d)$. On a subsequence (not relabeled) we thus find $u_n \wto u$ in $W^{r+2,p}(\Omega; \R^d)$ and easily deduce $L_{i,\mu}^{t_n} u_n \to L_{i,\mu}^{t} u$ in the sense of distributions. Hence, $L_{i,\mu}^{t} u = f$ and $t \in I_{i,\mu}$.

\emph{Claim 4: For $\mu \geq \mu_i$ the sets
\[I_{i,\mu} = \{t \in [0,1] \colon L_{i,\mu}^{t} \text{is onto}\}\]
are open in $[0,1]$.}

Let $t \in I_{i,\mu}$. Then $L_{i,\mu}^{t}$ is continuous, onto and one-to-one and therefore an isomorphism by the closed graph theorem. Set
\[\delta = \frac{1}{ 2 \lVert (L_{i,\mu}^{t})^{-1} \rVert (\lVert L_{i,\mu}^{1} \rVert + \lVert L_{i,\mu}^{0} \rVert)} \]
and let $s \in [0,1]$ with $\lvert s-t \rvert < \delta$. Let $f \in W^{r,p}(\Omega; \R^d)$ and let $u_0 = (L_{i,\mu}^{t})^{-1} f$. We have to find a $u \in W^{r+2,p}(\Omega;\R^d) \cap H^1_0(\Omega;\R^d)$ with $L_{i,\mu}^{s} u = f$ which is equivalent to finding a fixed point of
\[G_s(u) = (L_{i,\mu}^{t})^{-1} (L_{i,\mu}^{t}u - L_{i,\mu}^{s} u +f).\]
We claim that $G_s \colon \overline{B_r(u_0)} \to \overline{B_r(u_0)}$ is well defined and a contraction for $r = \lVert u_0 \rVert$. Indeed, since
\[\lVert L_{i,\mu}^{t} - L_{i,\mu}^{s} \rVert \leq \delta (\lVert L_{i,\mu}^{1} \rVert + \lVert L_{i,\mu}^{0} \rVert),\]
we find
\[\lVert G_s(u) - u_0 \rVert \leq \lVert (L_{i,\mu}^{t})^{-1} \rVert \lVert L_{i,\mu}^{t} u - L_{i,\mu}^{s} u \rVert \leq  2 r \delta \lVert (L_{i,\mu}^{t})^{-1} \rVert (\lVert L_{i,\mu}^{1} \rVert + \lVert L_{i,\mu}^{0} \rVert) \leq r\]
and
\[\lVert G_s(u) - G_s(v) \rVert \leq \lVert (L_{i,\mu}^{t})^{-1} \rVert \lVert L_{i,\mu}^{t} - L_{i,\mu}^{s} \rVert \lVert u-v \rVert \leq  \frac{1}{2} \lVert u-v \rVert.\]
Banach's fixed point theorem gives the desired result.

\emph{Claim 5: For $\mu \geq \mu_i$, we have $0 \in I_{i,\mu}$.}

This is just the (scalar) Laplacian in each component. This is a well known result. E.g., this is a special case of results in \cite{gilbargtrudinger}.

Since $[0,1]$ is connected, we have shown that $I_{i,\mu} = [0,1]$ for $\mu \geq \mu_i$. In particular, $L_{1,\mu}$ and $L_{2,\mu}$ are isomorphisms. The estimates follow from claim 2.
\end{proof}

\bibliographystyle{myalpha}
\bibliography{papers}

\end{document}